\documentclass[a4paper,12pt,reqno]{amsart} 
\usepackage{amssymb,amsmath,amscd,epsf,xypic,amsxtra}

\usepackage[pdftex]{graphicx}
\usepackage[all,knot,curve]{xy}

\def\B{\Cal B}
\def\P{\mathbb P}
\def\C{\mathbb C}
\def\h{\mathbb H}
 
\def\n{\mathbb N} 
\def\N{\mathbb N} 
\def\Q{\mathbb Q}
\def\r{{\mathbb R}}
\def\R{\mathbb R} 
\def\T{\mathbb T} 
\def\z{\mathbb Z} 
\def\Z{\mathbb Z} 

\def\I{\mathbb I} 
\def\c{\mathtt c}

\def\A{\mathcal A}

\def\ho{H\"older\ }
\def\bx{\qed}
\def\ep{\varepsilon}
\def\wt{\widetilde}
\def\wh{\widehat}
\def\bbar{\overline}
\def\goto{\rightarrow}

\def\Cal{\mathcal}
\def\cal{\mathcal}

\def\d{\text{d}}
\def\dd{\text{\rm d}}
\def\xx{x}
\def\sbar{\underline }

\renewcommand{\theequation}{\thesection.\arabic{equation}}  

\newtheorem{theo}{Theorem}[section] 

\newtheorem{lem}{Lemma}[section]

\newtheorem{cor}{Corollary}[section]

\newtheorem{prop}{Proposition}[section]

\newtheorem{defi}{Definition}[section]

\newcommand{\eqnsection}{

\renewcommand{\theequation}{\thesection.\arabic{equation}}
    \makeatletter
    \csname  @addtoreset\endcsname{equation}{section}
    \makeatother}
\eqnsection

\theoremstyle{plain}

\theoremstyle{definition}

\theoremstyle{remark}
\newtheorem{rem}{Remark}[section]

\newtheorem*{rem*}{Remark}

\newtheorem*{note*}{Note}
\newtheorem*{prop*}{Proposition}

\newtheorem{exam}{Example}

\newtheorem*{exam*}{Example}

\begin{document}
\title[Asymptotic self--similarity]
{Asymptotic self-similarity  and order-two ergodic theorems  for
renewal flows}

\date{March 19, 2013}

\author{Albert Meads Fisher} 

\address{Albert M. Fisher,
Dept Mat IME-USP,
Caixa Postal 66281,
CEP 05315-970
S\~ao Paulo, Brazil}
\urladdr{http://ime.usp.br/$\sim$afisher}
\email{afisher@ime.usp.br}

\author {Marina Talet}
\address{Marina Talet
C.M.I. Universit\'e d'Aix-Marseille
LATP, CNRS-UMR 6632, 
F-13453 Marseille Cedex 13, 
France}

\email{marina@cmi.univ-mrs.fr}

\thanks{
A. Fisher partially 
supported by CNPq, FAPERGS, FAPESP, COFECUB CAPES-CNRS 661/10, and the Brazil-France
agreement/CNPq-CNRS}

\thanks{
M. Talet partially 
supported by the Brazil-France agreement/CNPq-CNRS and FAPESP}

\define\bx{\qed}
\define\wt{\widetilde}
\define\wh{\widehat}

\define\z{\mathbb Z} 
\define\T{\mathbb T} 
\define\C{\mathbb C}
\define\Q{\mathbb Q}
\define\ep{\varepsilon}
\define\varep{\varepsilon}
\define\ho{H\"older\ }
\define\h{\mathbb H}
\define\n{\mathbb N}
\define\e{\mathbb E}
\define\goto{\rightarrow}
\define\E{\Cal  E}

\define\bbar{\overline}
\define\r{\mathbb R}
\define\s{\Cal S}

\keywords{
renewal
 process, Mittag--Leffler process, almost-sure invariance
principle in log density, scaling flow, Hausdorff measure, order--two ergodic theorem,
order--two density, cocycle, infinite measure ergodic theory}

\maketitle

\begin{abstract}We prove  a {\em log average almost-sure invariance
principle} ({\em log asip}) for renewal processes with positive
i.i.d.~gaps in the domain of attraction of an $\alpha$--stable law
with $0<\alpha<1$.
Dynamically, this means that renewal and Mittag-Leffler paths 
are forward asymptotic in the scaling flow, up to a time average.
This
strengthens the almost-sure invariance
principle in log density we proved in \cite{FisherTalet11b}.
The scaling flow is a Bernoulli flow on a probability space. 
We study a second  flow, the {\em increment
flow},  transverse to the scaling flow, which preserves an infinite
invariant measure constructed using singular cocycles.
A cocycle version of the Hopf Ratio Ergodic Theorem leads to 
  an order--two  ergodic theorem for the
  Mittag--Leffler  increment flow.
Via the  {\em log asip}, this  result then passes to a second
increment flow,
associated to the renewal process. As corollaries, we have
new proofs of
 theorems    of \cite{AaronsonDenkerFisher92} and of
 \cite{ChungErdos51},  motivated by fractal geometry.
\end{abstract}

\section{Introduction}

Let 
$(N_n)$ be a  renewal process 
with i.i.d.~positive integer gaps $(X_i)$ of distribution
function $F$ in the 
 domain of attraction of a completely asymmetric  $\alpha$--stable law
 with distribution function 
 $G_{\alpha}$, for $\alpha\in (0,1)$. Thus, there exists a 
 {\em normalizing sequence}   $(a_n)$ such that the following convergence in law holds:
\begin{equation}
\label{e:stableclt}
\frac{1}{a_n} \sum_{i=1}^n X_i\stackrel{\text{ law}} \longrightarrow G_{\alpha} \; \; 
\;\;\;\;\;\;\text{as $n \to \infty$.}
\end{equation}

For this range of $\alpha$, the  
expected gap  between events is infinite and 
the 
set $\Cal O$ of event occurrences has density zero.
Nevertheless, as we shall see,  $\Cal O$  has quite a bit of structure: it can be thought of as a
fractal integer set, for which one can calculate 
a ``dimension'' and a
``Hausdorff measure''. Moreover, the set  $\Cal O$ 
rescales to a dilation--invariant collection of fractal subsets of the
reals, 
which have  that same 
dimension and   average Hausdorff measure,
at both small and large scales.

\smallskip

 These observations  have interesting ergodic theoretic consequences.
The above renewal process  describes the times of return to a subset of finite
measure of a conservative ergodic infinite measure--preserving map,
the {\em renewal transformation} (which can be represented as a
countable state Markov chain with stationary infinite measure);  the transformation 
   renormalizes under scaling to an infinite measure preserving flow,
the 
{\em increment flow} of the Mittag--Leffler process. 
This flow 
 is transverse to the 
 dynamics of  renormalization, given by  the scaling flow of index
 $\alpha$, which due to the
 self--similarity of the Mittag--Leffler process  preserves a  probability measure. 
The renormalization is expressed in a commutation relation linking
the pair of flows, identical to that shared by 
the
geodesic and horocycle flows on Riemann surfaces of infinite area. As a
consequence of the  renormalization approximation, plus ergodicity of
both flows, we derive
an order--two ergodic theorem for the renewal flow, mirroring results
for the infinite measure horocycle flow. This type of result should be
considered an infinite--measure analogue of the Birkhoff ergodic
theorem statement that ``time average equals space average'', seen
through the lens of Hausdorff measure and fractal geometry \cite{Fisher92}.

\smallskip

But first
let us   be more precise  regarding the 
fractal--like structure of the set $\Cal O$. One can show that 
for a.e.~path $N_n$, the ratio $  (\log N_n)/ \log n$ converges to $\alpha$,  the 
{\em dimension} of the integer set $\Cal O$  as defined in \cite{BedfordFisher92}.

\smallskip

This result, noted by Chung and Erd\"os (in Theorem 7 of \cite{ChungErdos51}) gives
$n^\alpha$ as a first estimate for the growth of $N_n$.
A better approximation is given by $\wh a_n=
1/(1-F(n)),$  a regularly varying sequence of index $\alpha$, as 
 $N_n /\wh a_n$  converges in law to the {\em Mittag--Leffler}
 distribution of index $\alpha$, which has distribution function
${\cal M}_{\alpha}(x)= (1-G_{\alpha}(x^{-\frac{1}{\alpha}}))$; see
Theorem 7 of
\cite{Feller49}.

Writing $Y_i= 1$ when an event occurs and  $Y_i= 0$
otherwise, so $N_n= \sum_{i=1}^nY_i$, and defining the {\em return sequence} $\bbar a_n =\e(N_n)$,
 the expected number of events up to time $n$, one has the following
further result of Chung and Erd\"os (Theorem 6 of \cite{ChungErdos51}):
\begin{equation}\label{e:ChungErdos}
\lim_{k\to \infty} \frac{1}{\log \bbar a(k)} \sum_{n=1}^k \frac{Y_n
  }{\bbar a_n} =1, \text{ a.s.}
\end{equation}
As noted in Proposition 1 of \cite{AaronsonDenkerFisher92}, when $\bbar a_n$ is regularly
varying of index $\alpha\in (0,1)$ (which will be the case here)   then  \eqref{e:ChungErdos}
is equivalent to:
\begin{equation}\label{e:ADF}
\lim_{k\to\infty} \frac{1}{\log k} \sum_{n=1}^k\frac{N_n}{\bbar a_n}
\frac{1}{n}=1, \text{ a.s.}
\end{equation}

The different normalizations $a_n, \wh a_n, \bbar a_n$ can be
unified. First, as shown by L\'evy, a normalizing sequence $a_n$ for  
\eqref{e:stableclt}
is necessarily regularly varying of index  $1/\alpha$.  Now  
as in
\cite{FisherTalet11a}, one can find an especially nice normalizing
sequence: we construct from $F$ an increasing
$C^1$ function  $a(\cdot) $ with regularly varying derivative so  its
inverse $\wh a\equiv a^{-1}$  is, at integer values, asymptotically  equivalent to Feller's sequence $\wh a_n$,
written throughout $\wh a(n)\sim \wh a_n$ (i.e.~their ratio goes to 1).
There exists, furthermore, a constant ${\mathtt c}\in
(0,1)$ such that $\bbar a_n \sim {\mathtt c}\cdot\wh a(n)$; see
Corollary \ref{c:logaveragerentf}.

\smallskip

Replacing $\bbar a_n$ by $\wh a(n)$
in the above averages exchanges $1$ for  ${\mathtt c}$ as the limiting
value, and equation \eqref{e:ADF} then 
 admits a fractal geometric
 interpretation: the limit gives  the  
``Hausdorff measure''   of the integer set $\Cal O$ for the
  gauge function $\wh a(\cdot)$ (see Proposition \ref{prop1}) with the number $\mathtt c$
equalling
the order--two  density \cite{BedfordFisher92} of the limiting
fractal sets of reals.

To explain this more precisely, we move to the more general setting of biinfinite
renewal
processes with  real gaps.
For $(X_i)_{i\in\z}$  an i.i.d.~sequence with distribution supported on $(0,+\infty)$,
 defining $(S_n)_{n\in\z}$ by $S_0=0$, 
\[
S_n= 
\left\{ \begin{array}{ll} 
&\sum_{i=0}^{n-1}X_i \text{ for } n>0  \\
-&\sum_{i=n}^{-1}X_i \text{ for } n<0\\
\end{array} \right.
\]
and
$\bbar S(t)= S_{[t]}$, then for $\wh {\bbar S}  (t) = \inf\{s: \bbar S(s)>t\}$ the generalized inverse, 
we call $\bbar N= \wh{\bbar S}-1$  the {\em two--sided renewal
process} with gaps $(X_i)$.  Thus, $\bbar N(0)= 0$ and for
$t\geq 0$, 
$\bbar N(t)$ is the total number of events which occur up to and
including time
$t$, excluding the initial event at time  $0$. For the special case of integer gaps, 
then 
$\bbar N(0)=0$, and
\begin{equation}
  \label{eq:N_n}
   \bbar N(n)= 
\left\{ \begin{array}{ll} 
&\sum_{i=1}^{n}Y_i \;\; \;\;\;\;\text{ for } n>0  \\
-&\sum_{i=n+1}^{0}Y_i \;\;\; \text{ for } n<0.\\
\end{array} \right.
\end{equation}
This agrees with the usual definition of integer gap renewal process $N_n$ for
$n\geq 0$, so we use the above equation to extend to all $n\in\Z$,
setting $N_n\equiv \bbar N(n)$.
Both the one-- and two--sided
processes appear  in this paper.  

As in
\eqref{e:stableclt}, we assume that the distribution of $X_i$ is in the domain of
attraction of 
$G_{\alpha}$.

\smallskip

We recall the definition of the Mittag--Leffler process; see also \cite{FisherTalet11b}.
Let $D= D_{\r^+}$ denote {\em Skorokhod path space}: the collection of
all c\`adl\`ag (right--continuous with left limits) real--valued functions on
$\R^+$, equipped with
 the Skorokhod $J_1$--topology.
For $Z$  a completely asymmetric  stable process of index  $\alpha\in
(0,1)$, represented on 
$D$  with  Borel probability measure $\nu$, a.e.~$Z$ is increasing and 
the Mittag--Leffler paths  $\wh Z(t)$ of parameter $\alpha$ are the
generalized inverses 
of the stable paths:
\begin{equation}\label{e:geninv}
\wh Z(t)= \inf\{s:\, Z(s)>t\}.
\end{equation}

Denoting by  $\wh \nu$  the corresponding measure   for the
Mittag--Leffler process, 
the {\em scaling flow}  on $D$ of index $\alpha$
 defined by
 \begin{equation}
   \label{eq:scaling flow}
(\wh\tau_t \wh Z)(x)= \frac{\wh Z(e^t x)}{e^{\alpha t}}  
 \end{equation}
is ergodic, in fact  is a Bernoulli flow of infinite
entropy, see Proposition \ref{p:nextcommreln}. 

Our first result states:
\begin{theo}\label{t:main2} 
{\em (A log average   almost-sure invariance
principle for renewal processes, $\alpha\in(0,1)$)}
Let $(X_i)_{i\geq 0}$ be an i.i.d.~sequence of a.s.~positive random variables with common 
distribution function $F$ in the domain of attraction of an $\alpha$--stable
law. Then there exists a $C^1$, increasing, regularly
varying function $a(\cdot)$ of index  $1/\alpha$  with
regularly varying derivative,  such that $a(n)$ gives a
normalizing sequence of \eqref{e:stableclt}, and such that, 
for the regulary varying function of index one
$h(\cdot)\equiv a^{\alpha}(\cdot)$, 
then 
$\bbar N$ and  a Mittag--Leffler process $\wh Z$  of index $\alpha$
can be redefined to live on the same probability space so as to
satisfy:

\begin{equation}\label{main}
\lim_{T\to \infty} \frac{1}{\log T} \int_1^T \frac{|h({\bbar N}(t))-
  \wh Z (t) |}{t^{\alpha}}\; 
\frac{\dd t}{t} =\lim_{T\to \infty} \frac{1}{\log T} \int_1^T 
\frac{||h\circ{\bbar N}- \wh Z ||_{[0,t]}^{\infty}}{t^{\alpha}}\; \frac{\dd t}{t} =0\;\;\;\; a.s.
\end{equation}

\noindent
where $||f-g ||_{[0,t]}^{\infty}=\sup_{0\leq s\leq t}|f(s)-g(s)|$. 

\smallskip

Equivalently, in dynamical terms, 
for $\wh \tau_t$ the scaling flow of index  $\alpha$,
there exists a joining of  the laws of the Mittag--Leffler and renewal processes $\wh \nu$ and $\wh \mu $, such that 
for a.e.~pair $(\wh Z, \bbar N)$, 
\[
\lim_{T\to \infty}\frac{1}{T}\int_0^T d_1^u(\wh\tau_t(\wh Z),\wh\tau_t(h\circ\bbar N))\,\dd t= 0,\]
\noindent
with $d_1^u (f,g)=||f-g||_{[0,1]}^{\infty}$; we say that $h \circ \overline N$ is in the {\em Ces\'aro--average} 
$d_1^u$--stable manifold $W^s_{CES}(\wh Z)$.
\end{theo}

We mention the importance of finding  a representative
 for $a(\cdot)$ which is $C^1$ increasing
 with regularly varying derivative:
  such a function preserves log averages, see
 Proposition 2.5 of \cite{FisherTalet11a}, and that plays a key role in the proofs.

An important element in the proof of Theorem \ref{t:main2} is
a weaker statement proved in Theorem 1.1 of \cite{FisherTalet11b},  an {\em almost--sure invariance principle in log density, {\it asip} (log)}: 
that 
for
  a.e.~pair $(\bbar N, \wh Z)$ we have:
\begin{equation}\label{hNbar}
||h \circ \overline N -\wh Z||^\infty_{[0,T]} =o(T^{\alpha})\, \; (\log),\end{equation}
\noindent 
where $(\log)$ means off a set ${B}\subseteq \r^+$ of logarithmic density ($\equiv \lim_{T\to\infty} \frac{1}{\log T} \int_1^T \chi_{{B}} (t) \d t/t$) equal to zero, for $\chi_{{B}}$  the indicator function of the set $B$.

\medskip

\smallskip
Now the Birkhoff Ergodic Theorem for the scaling flow $\wh\tau_t$ of the
Mittag--Leffler process, together with  a  logarithmic change of
variables,  tells us that $\wh \nu$--a.s.
\begin{equation}\label{e:logaverageML}
\lim_{T\to\infty} \frac{1}{\log T} \int_1^T \frac{\wh Z(t)}{t^\alpha}\frac{1}{t}\, \d t=
\lim_{T\to\infty} \frac{1}{T} \int_0^T \varphi(\wh \tau_t(\wh Z)) \, \d t=\e(\varphi)=\e({\wh Z}(1)),
\end{equation}
where $\varphi(f)= f(1)$ defines an $L^1$- observable
on  the probability space $(D, \wh\nu)$.

\smallskip

Combining this with the following further result from  \cite{FisherTalet11b}: 
 for the inverse
function $\wh a\equiv
a^{-1}$, 
\begin{equation}\label{hN}
\biggl|\biggl|\frac{\bbar N(e^t \cdot)}{\wh a(e^t)}- \wh \tau_t( \wh Z)\biggr|\biggr|^\infty_{[0,1]} =
o(1)\,  \;\;\;\;
\text{ a.s. } (\log),
\end{equation}
which was shown to follow from \eqref{hNbar},
we  then   prove:

\begin{cor}\label{c:ChungErd}
Under the assumptions and notation of Theorem \ref{t:main2}, $\wh \mu$--a.s.
\begin{equation}\label{flipu}
\lim_{T\to \infty} \frac{1}{\log T} \int_1^T \frac{\bbar N(t)}{\wh a(t)} \frac{\dd t}{t}=\lim_{T\to \infty} \frac{1}{\log T} \int_1^T \frac{h({\bbar N}(t))}{t^{\alpha}}\; \frac{\dd t}{t} =\e(\wh Z(1)).
\end{equation}
\end{cor}
 To explain the connection of the theorem and corollary  with fractal geometry, we note that the paths  $\wh Z(t)$ are
nondecreasing and continuous, with a nowhere dense set $C_{\wh Z}$
of points of increase. By a result
of Hawkes, 
for the gauge function
\begin{equation}\label{e:Hawkesgauge}
\psi(t)= t^\alpha(\log\log \frac{1}{t})^{1-\alpha},
\text{ then for  }
c_\alpha= \frac{\alpha ^{1-\alpha}(1-\alpha) ^\alpha}{\Gamma(3-\alpha)},
\end{equation}
$\wh Z(t)/c_\alpha$ is the
distribution function of 
 the Hausdorff measure $H_\psi$ restricted to the set $C_{\wh Z}$; see \S \ref{s:range}.
 That is, 
a.e.~Mittag--Leffler path $\wh Z(t)$  is the 
$c_\alpha H_\psi$--Cantor function for the  $\alpha$--dimensional
fractal set $C_{\wh Z}$, and  the content of Theorem \ref{t:main2}
is a precise description of the convergence of the discrete set of events of
the renewal process to this fractal set of reals.

\smallskip

An application of the Ergodic Theorem then allows us  to relate the
limiting constant of 
\eqref{flipu} to small--scale fractal geometry:

\begin{prop}\label{prop1} For $\psi $ as in \eqref{e:Hawkesgauge}, 
the constant $\c\equiv \e(\wh Z(1))= \int f(1) \d\wh \nu(f)$ is 
equal to $c_\alpha $ times 
 the one--sided order--two density of 
the set of points of increase of the process $\wh Z$ with respect to $H_\psi$: for $\wh \nu$--a.e.~$\wh Z$, for $H_\psi$-a.e~point $x\in C_{\wh Z}$, 
\[
\c=  c_\alpha \cdot
\lim_{T\to\infty}\frac{1}{T}\int_0^T\frac{H_\psi(C_{\wh Z}\cap[x,
  e^{-s}])}{e^{-s\alpha}}\dd s = \frac{\sin\pi \alpha}{\pi \alpha}
.\] 
\end{prop}

Now we turn to the analogy with the geodesic and
stable horocycle flows on the unit tangent bundle of an infinite area  Riemann surface
uniformized by a finitely generated Fuchsian group. These flows are ergodic, and preserve, respectively, a 
probability measure, the Patterson-Sullivan measure
\cite{Patterson76} \cite{Sullivan79}, and a  related
infinite measure, first studied by Kenny \cite{Kenny83}; 
see  
\cite{Fisher04-2}. Moreover, they satisfy the commutation relation
 \begin{equation}
  \label{eq:comm reln}
g_t\circ  h_s= h_{e^{-t}s}\circ g_t.
\end{equation}
Here,
the scaling flow $\wh\tau_t$ on Mittag--Leffler paths will play
the role of the geodesic flow, 
with the part of the 
horocycle flow  taken by the {\em increment flow}  $\bbar\eta_t: f(x)\mapsto f(x+t)- f(t)$
 on the two--sided Skorokhod path space $D= D_\r$; the scaling flow
 acts on the two--sided paths again by equation \eqref{eq:scaling
   flow}, with the pair   $\wh\tau_t, $
 $\bbar\eta_s$ 
obeying the same relation as for $g_t$, $h_s$.

From the renewal process 
we  define the {\em renewal flow}, also  
realized as
an increment flow on $D$ but with
 a different infinite invariant measure. 
The  commutation relation 
 can be thought of as stating that  the geodesic flow 
 renormalizes the horocycle flow to itself, and that $\wh\tau_t $
 renormalizes the Mittag--Leffler increment flow to itself.
Then Theorem  \ref{t:main2}  says, essentially, that
 the renewal increment flow renormalizes to the Mittag--Leffler 
increment flow in the limit, by applying the scaling flow to the
joined pair 
of paths.

To state  our next results, which describe the ergodic theoretic consequences,
we  recall these definitions;  see
\cite{Aaronson97}. 
A measure preserving flow $\tau_t$ of a possibly infinite measure space 
$(X, \mu)$ is {\em ergodic} iff any {\em invariant}  set $A$ has 
$\mu A= 0 $ or $\mu A^c= 0$; the  flow is {\em recurrent}
iff almost every  point returns to a subset of  positive measure $A$
for arbitrarily large times, an equivalent notion being  that the flow is  {\em conservative}.

The invariant measures  for
both the Mittag--Leffler and renewal increment flows are
constructed via {\em cocycles} over the flows, see \S\S \ref{ss:cocycles}
and \ref{ss:incrementML}. These  are both conservative ergodic infinite measure flows
(Propositions \ref{p:nextcommreln} and
\ref{p:renewal flow}).
 Our conclusions 
are also most naturally  stated and proved using cocycles.
In the statement,
the {\em integral} of a  cocycle $\Phi$ over the flow $(D_\R, \bbar \nu, \bbar \eta_t)$ is 
$\I(\Phi)\equiv \frac{1}{t}\int_D \Phi(x,t)\d \bbar \nu(x)$; this does not depend on $t$,
 and agrees with the usual 
notion of  integral of a function in the special case that
the cocycle is generated by a  function (Proposition \ref{p:expectedvalue}).

We prove:
\begin{theo}(Order--two ergodic theorems for Mittag--Leffler and renewal
increment flows)\label{t:logaveragetheorem}

\smallskip
\item{(i)}
Let $\Phi(x,t)$ be a cocycle over the increment flow $(D_\R, \bbar \nu, \bbar \eta_t)$ on Mittag--Leffler
 paths 
which is measurable,  of
local bounded variation in $t$, and with $\I(\Phi)$
finite.
Then for  $\bbar \nu$--a.e.~ Mittag-Leffler path $\wh Z$, and for  $\c$ as above, 
\begin{equation}\label{e:MLlogavgthm}
\lim_{T\to \infty} \frac{1}{\log T} \int_1^T \frac{\Phi(\wh Z,t)}{t^\alpha}\; \frac{\dd t}{t}=\c \; \I(\Phi) .
\end{equation}

\smallskip

\item{(ii)}  Assuming the gap distributions of Theorem \ref{t:main2},
let $\Phi(x,t)$ be a cocycle over the increment flow on renewal paths
$(D_\R, \bbar \mu, \bbar\eta_t)$ which is measurable, of local bounded
variation in $t$, and with
$\I(\Phi)$
finite.
Then 
  for $\bbar \mu$--a.e.~renewal path $\bbar N$, and for $\wh a(\cdot)$ and $\c$ as above, we have:
\begin{equation}\label{e:logAv}
\lim_{T\to \infty} \frac{1}{\log T} \int_1^T \frac{\Phi(\bbar N,t)}{\wh a(t)}\; \frac{\dd t}{t}=\c \; \I(\Phi).
\end{equation}
\end{theo}

One of the steps in the proof is to show, in Proposition \ref{p:stable
  manifolds}, that the increment flow
satisfies, for both the 
Mittag--Leffler and renewal measures, a property akin to the defining
property of the classical
horocycle flow: that  the stable manifolds of the geodesic
flow are preserved. In our setting the stable manifolds are weakened
to 
 the  Ces\`aro--average stable manifolds, leading to this  precise statement:
for $\wh \nu$--a.e.~$\wh Z$, for all $t\in \R$, $\eta_t \wh Z$ is  in
$W^s_{CES}(\wh Z)$, while
 with respect to 
the joining of   Theorem \ref{t:main2}, for  a.e.~pair 
$(\wh Z, \bbar N)$,  for all $t\in \R$, $\eta_t (h\circ \bbar N)$ is  in $W^s_{CES}(\wh Z)$.

\medskip

Specializing next to the case of integer gaps with which we began the
paper, we consider a countable-state Markov chain. The ergodic theory
model is the following:  we give
the {\em alphabet}
$\A$ (this is the collection of states)
  the discrete topology
and the biinfinite path space
$\Pi= \Pi_{-\infty}^{+\infty} \A$ the product topology, with $\B$  the
Borel $\sigma$--algebra of $\Pi$. 
The left shift map 
$\sigma$ acts on   $\Pi$ by
$(\sigma(\xx ))_i= x_{i+1}$, thus
 sending  $\xx= (\dots x_{-1} .x_0x_1\dots)$ to $ (\dots
x_{-1} x_0.x_1\dots)$. 
We  are given a row--stochastic matrix $P= (P_{ab})$, i.e.~such that
$\sum_{b\in\A}
P_{ab}=1$ for each $a\in\A$, 
and an invariant row vector 
$\wh\pi= (\wh\pi_a)_{a\in\A}$, so  $\wh\pi=\wh\pi P$, with $\rho$  the
shift--invariant 
Markov measure determined by $P$
and $\wh\pi$. Then 
$(\Pi, \B, \rho,\sigma)$ is a measure--preserving
transformation (of infinite measure iff $\wh\pi$ has
infinite mass), with  $x=(x_i)_{i\in\Z}$ 
  a path of the time biinfinite
stationary Markov chain. 

Choosing a state ${\mathtt a}\in\A$, let us assume that
$\rho$ is normalized so that for
$A= \{x:\, x_0= {\mathtt a}\}$, 
$\rho(A)=1$. For $i\in \Z$ define $Y_i=
\chi_A(\sigma^i(x))$, so $Y_i=1$  when the event ${\mathtt a}$
occurs.  Determining $N_n$ by equation \eqref{eq:N_n}, with $N_n$
replacing $\bbar N(n)$,
this defines  the {\em occupation
time process} of the state ${\mathtt a}$.

We show:
\begin{cor}\label{c:logaveragerentf} 
Given a stationary  infinite measure conservative ergodic  Markov chain 
$(x_i)_{i\in \Z}$ taking values in a countable alphabet $\A$, so
$x= (x_i)_{i\in \Z}\in \Pi= \Pi_{-\infty}^\infty\A$,  assume that
for some state ${\mathtt a}\in\A$, the shift--invariant
measure $\rho$ on $\Pi$ is normalized so that $\rho[x_0={\mathtt a}]=1$.
Then  for $ N_n$ the occupation time process of ${\mathtt a}$, 
 these conditions are equivalent:

\smallskip

\item{$(a)$} the  return sequence $ {\bbar a_n}\equiv  \e(N_n)$ is regularly varying
  of index  $\alpha\in (0,1)$;

\item{$(b)$} the distribution function $F$ of return times  to  state ${\mathtt a}$ is in
    the domain of attraction of $G_{\alpha}$.

\smallskip

\noindent
Suppose that $(a)$ (or $(b)$) holds. Then letting  $\sigma$  denote the left shift map on $(\Pi,\rho)$, we have that 
for any $\varphi\in L^1(\Pi, \rho)$, 
writing $S_n\varphi= \sum_{i=0}^{n-1}\varphi(\sigma^ix) \text{ for } n>0 ,$
for $\rho$--a.e.~$x$, 
\begin{eqnarray}\label{e:secondrenewallogavg}
(i)\; \; &\lim_{k\to \infty}& \frac{1}{\log k} \sum_1^k \frac{S_n\varphi (x)}{\wh
  a(n)} \frac{1}{n}=\c \; \int_\Pi\varphi\dd\rho,\\
\label{e:ChungErdosb}
(ii)\;\; &\lim_{k\to \infty}& \frac{1}{\log \wh a(k)} \sum_{n=1}^k \frac{\varphi
  \circ \sigma^n(x)}{\wh a(n)} =\c \; \int_\Pi\varphi\dd\rho,
\end{eqnarray}
for $\wh a(\cdot)$ and  $\c$  as above.

\noindent
$(iii)$ Moreover, $\bbar a_n\sim {\mathtt c}\cdot\wh a(n)$  hence 
\eqref{e:secondrenewallogavg},
\eqref{e:ChungErdosb} are equivalent to results of
\cite{AaronsonDenkerFisher92} 
for this case, with \eqref{e:ChungErdosb} implying a theorem of \cite{ChungErdos51}.
\end{cor}

In particular, the corollary applies to the {\em renewal
  transformation}, the infinite measure transformation built from a 
renewal process with integer gaps; one of the models for this map is a
countable state Markov chain.
See  \S \ref{ss:integergaps}.

\medskip

The outline of the paper is as follows. In \S
\ref{1} we prove Theorem \ref{t:main2}. This  follows from a key
result,  Lemma \ref{NtoS}, proved in \S \ref{ss:proof of main}.
Its proof hinges upon two preparatory lemmas regarding normalized
partial sums of i.i.d.~variables in the domain of attraction of the
stable law $G_\alpha$, stated and proved in \S
\ref{l:tolemmas},
as well as an idea of Ibragimov and Lifshits. In \S \ref{ss:follows
  from lema} we show that  Lemma \ref{NtoS} yields
Theorem \ref{t:main2}. 
 Lastly, the proof of Corollary  \ref{c:ChungErd} is presented
at the end of \S \ref{ss:follows
  from lema}.

Then we turn to the ergodic theory:  the construction
of the flows and measures,  
and the proof of Theorem \ref{t:logaveragetheorem}. To carry this out,  in \S
\ref{s:cocycles} we  develop some suitable machinery in 
abstract ergodic theory regarding cocycles and dual flows,, and prove a cocycle version of the Hopf ratio
ergodic theorem. In \S \ref{s:dual flows} we study the scaling and
increment flows and their duals. In \S \ref{s:measuredualincrement} we
use cocycles to construct our invariant measures. In \S \ref{ss:integergaps} we
describe several models for the renewal transformation and  explain
the connection with the renewal flow: it can be seen as 
the suspension flow over the transformation. In \S \ref{s:log average ergodic}
we first prove a general order--two ergodic theorem, valid for self--similar processes which are dual to 
processes with stationary increments. We then show in \S \ref{ss:stable manifolds} that the increment
flow behaves like a horocycle flow, and in \S \ref{ss:proofofthm} we
finish the proofs of 
Theorem \ref{t:logaveragetheorem} and  Corollary \ref{c:logaveragerentf}.
The identification of the constant $\c$ in terms of fractal geometry is given in Lemma 
\ref{p:identifyingconstant}, with  Proposition \ref{prop1}  proved in \S \ref{s:range}.

\smallskip

\section{Proof of  the log average {\em asip}, Theorem \ref{t:main2}} 
\label{1}

We recall from \S 2 of \cite{FisherTalet11a} some background and
notation which will be used throughout the 
paper; see also 
\cite{Feller71} pp 312-315 and 570. For  $\alpha\in (0,1)\cup (1,2]$
and $\xi\in [-1,1]$, 
   a random variable $X$ has {\em stable law}
$G_{\alpha, \xi, \kappa,\theta} $ if  its characteristic function
(i.e.~Fourier transform) is
\begin{equation}
  \label{eq:charfunctionstable}
   \e(e^{iwX})=  \exp\left(i\theta w+\kappa\cdot \, \frac{\Gamma(3-\alpha)}{\alpha(\alpha-1)}\, 
|w|^\alpha \bigl( \cos\frac{\pi \alpha}{2} - \frac{w}{|w|}i \xi
 \sin 
\frac{\pi \alpha}{2}\bigr) \right).
\end{equation}

There is a unique 
process $Z$ with stationary independent increments and with 
$G_{\alpha,\xi}\equiv G_{\alpha, \xi, 1,0}$ the law for $Z(1)$; $Z$ is self--similar of index
$1/\alpha,$ thus the law of $Z(t)$
is $G_{\alpha,\xi, t,0}.$ Now $G_{\alpha,\xi}$ has support on
$(0,+\infty)$ if and only if $\alpha\in (0,1)$ and $\xi=+1$, in which case it
is a (positive)
{\em completely asymmetric} law. The corresponding 
 process  $Z$ 
 is known as the {\em stable subordinator} of index $\alpha$.

\smallskip

 A distribution function $F$ supported on $(0,\infty)$  belongs to the domain of 
attraction (see \eqref{e:stableclt}) of $G_\alpha\equiv G_{\alpha,1}$
with $\alpha\in (0,1)$ if and only if $1-F$ is regularly varying of
index $-\alpha$, that is $1-F(x)=x^{-\alpha}\, l (x)$
with $l(\cdot)$ some slowly varying function  (i.e.~for all $x>0$, $
l(tx)/l(t)\to 1$ as $t\to\infty$). In this paper regular variation
always means at $+\infty$ unless indicated otherwise.

\smallskip

 The proof of  Thm.~\ref{t:main2} is carried out in several steps. We begin with  some preparatory material.

\subsection{Two lemmas}\label{l:tolemmas}
\begin{lem}\label{convmoments} Let $(X_i)_{i\geq 0}$ be an
  i.i.d.~sequence of a.s.~positive variables with common distribution function $F$ in
  the domain of attraction of $G_{\alpha}\equiv G_{\alpha,1}$ for
  $\alpha\in (0,1)$. So by \eqref{e:stableclt} there exists a regularly
  varying sequence  $a_k$ such that $S_k/a_k$ converges in law to $Z(1)$, with $\alpha-$stable distribution.  

\smallskip

\item{$(i)$} (de Acosta and Gin\'e, see \cite{deAcostaGine79}, page 225) For any $0\leq \beta<\alpha$, we have
\[
\lim_{k\to\infty} \e\left(\left(\frac{S_k}{a_k}\right)^{\beta}\right) =\e(Z^{\beta} (1)) <\infty.
\]

\smallskip

\item{$(ii)$}  For any $p>0$, we have
\begin{equation}\label{lulu}
\lim_{k\to\infty} \e\left(\left(\frac{a_k}{S_k}\right)^p \right) =\e\left(\frac{1}{Z^p (1)}\right)=\e(\wh Z^{p/\alpha} (1))<\infty.
\end{equation}
Equivalently, 
\begin{equation}\label{l:key}
\lim_{k\to\infty}\int_0^1\P(S_k \leq ta_k) \,\frac{\dd t}{t^{1+p}} =
\int_0^{1}\P(Z(1) \leq t) \, \frac{\dd t}{t^{1+p}}<\infty.
\end{equation}
\end{lem}

We shall also need the following

\begin{lem}\label{correlations} Under the assumptions of Theorem \ref{t:main2}, we set
\begin{equation}\label{xi_k}
\xi_k (t)= \chi_{S_k\leq t a(k)} - \P(S_k \leq  t a(k)),\;\; t\geq 0.
\end{equation} 
Let $\epsilon(\cdot)$ be a positive function which goes to 0 at
infinity. Denoting by 
$\wh a (\cdot)$ the inverse of $a(\cdot)$, set  $l_k= \wh a(\epsilon(k)a(k))$ and let $l$ and $k$ be two positive integers such that $l\leq l_k< k$.  Then we have 
\[
\e(\xi_k (t)\xi_l(t))\leq c\; \P\biggl(\frac{S_l}{a(l)}\leq t\biggr) \;
\biggl(\frac{a(l)}{t \varep(k)a(k)}\biggr)^{\alpha^{-}} + \; \P\biggl(t(1-\varep(k)) \leq \frac{S_k}{a(k)}  \leq t(1+\varep (k))\biggr),
\]
for $l$ large, with $c$  some positive constant and  $0<\alpha^{-}<\alpha$. 
\end{lem}

We now move on to the proofs of these lemmas. 

\begin{proof}[Proof of Lemma \ref{convmoments}, $(ii)$] Since $F$
  belongs to the domain of attraction of  $G_{\alpha}$, we have
\begin{equation}\label{F}
1-F(x)=x^{-\alpha}\; L(x),
\end{equation}

\noindent 
for $L$ some slowly varying function. Now by  \eqref{e:stableclt}, for all $p>0$  we have that
$(a_k/S_k)^p$ converges in law to $1/Z^p(1)$. From page 32 of
\cite{Billingsley68}, the convergence of moments  $(ii)$ will follow
from the uniform integrability of $\left((a_k/S_k)^p\right)_{k\geq
  k_o}$ for 
$k_0$ large, and  it is  enough to check that for some $p'>p$
\begin{equation}\label{ui}
\sup_{k\geq k_0} \e\left(\left(\frac{a_k}{S_k}\right)^{p'} \right) <\infty.
\end{equation}
For any positive random variable $X$,  using Fubini-Tonelli we have $\e(X)=  \int_0^{\infty} \P(X\geq t)\; \d t,$ so
\begin{equation}\label{pookie}
\e\left(\left(\frac{a_k}{S_k}\right)^{p'} \right) =p'\, \left( \int_0^1+\int_1^{\infty} \right) \P(S_k \leq t a_k)\; \frac{\d t}{t^{1+p'}}.
\end{equation}
The integral over $[1,\infty)$ is bounded. We  find an upper bound for  that over $[0,1]$, writing
\[
\P(S_k \leq ta_k)\leq \P(\max_{1\leq i\leq k} X_i \leq ta_k)= \left(\P(X_1\leq ta_k)\right)^k= \bigl(F(ta_k)\bigr)^k,
\]
\noindent 
as the $X_i$ are i.i.d.~with common distribution function $F$. From \eqref{F}, it follows that $\forall t>0$ and $k$ large enough,
\[
\P(S_k \leq ta_k)\leq \exp\left(k\log \left(1-\frac{L(ta_k)}{t^{\alpha} a^{\alpha}_k}\right)\right)
\leq \exp\left(-k   \; \frac{L(ta_k)}{t^{\alpha} a^{\alpha}_k}
\right),\]
\noindent 
where we have used the fact that $L(u)= o(u^{\alpha})$. Now,  $a^{\alpha}_k\sim k L(a_k)$ and by  Potter's Theorem (see \cite{BinghamGoldieTeugels87}, page 25) for any choice of  $A>1$ and $\delta>0$, 
\begin{equation}\label{potter}
\frac{L(y)}{L(x)} \leq A\, \max\left((y/x)^{\delta}\, ,
  (y/x)^{-\delta} \right), \; {\text{for}}\; x,y \;\; {\text {large\, enough.}}
\end{equation}
\noindent Taking $x=ta_k$ and $y=a_k$, this gives a lower bound  for $L(ta_k)/L(a_k)$.  Thus,
\begin{equation}\label{lebesgue}
\P(S_k \leq ta_k)\leq \exp
\left(-\frac{c}{t^{\alpha-\delta}}\right)\;\; \;\;(t\in(0,1)),
\end{equation}
\noindent for some $c>0$, and so the
integral over $[0,1]$ in \eqref{pookie} is finite. This finishes the
proof of \eqref{ui}.

\smallskip

From the definition of $\wh Z$ (see
\eqref{e:geninv}), the fact that
$Z$ is the generalized inverse of $\wh Z$, the continuity
of $\wh Z$, and finally the self--similarity of $Z$, we have:
\begin{equation}\label{popo}
\forall t\geq 0,\;\; \P(\wh Z(1)\leq t)=\P(Z(t)\geq 1)=\P(t^{1/\alpha}\, Z(1)\geq 1)=
\P\left(\frac{1}{Z^{\alpha}(1)} \leq t\right).
\end{equation}
Thus $\wh Z(1)$ has the same law as $1/Z^{\alpha}(1)$, and as  $\wh Z(1)$ has finite moments of all order (see \cite{BinghamGoldieTeugels87},  page 337) we are done with the proof of \eqref{lulu}. Lastly, \eqref{lebesgue} together with Lebesgue's Dominated Convergence Theorem implies \eqref{l:key}, completing the proof of $(ii)$. 
\end{proof}

\begin{proof}[Proof of Lemma \ref{correlations}] From the definition of $\xi_{\cdot} (t)$, we get:
\[
\e(\xi_k (t) \xi_l (t))= \P(  S_l \leq t a(l); S_k \leq t a(k))- \P( S_l \leq t a(l))\; \P( S_k \leq t a(k)), 
\]

\noindent
with $\P(A \,; \,B)$ standing for $\P(A\cap B)$. 

\smallskip

Since $l<k$, $S_l$ is independent of $S_k - S_l$; moreover since the $X_i$ are a.s.~positive we have $S_l < S_k$ a.s. Thus $\e(\xi_k (t) \xi_l (t))$ can be rewritten as
\[
\int_0^{t a(l)} \P(S_k -S_l \leq ta(k)-x)\; \d \P_{S_l} (x) - \P( S_l \leq t a(l)) \; \int_0^{ta(k)} \P(S_k -S_l \leq ta(k)-x) \; \d \P_{S_l} (x) 
\]

\noindent 
where $\P_{S_l}$ denotes the law of $S_l$ or also the pushed--forward measure $S_l^{*} (\P)$. Splitting the last integral $\int_0^{ta(k)}$ into $\int_0^{ta(l)} + \int_{ta(l)}^{ta(k)}$, this equals
\[
\P(S_l > ta(l)) \int_0^{t a(l)} \P(S_k -S_l \leq ta(k)-x) \d \P_{S_l} (x) - 
\P( S_l \leq t a(l))\int_{ta(l)}^{ta(k)} \P(S_k -S_l \leq ta(k)-x) \d \P_{S_l} (x).
\]
Now, since  $a(l_k)=\varep(k) a(k)\leq a(k)$ for $k$ large then 
\begin{eqnarray*}
\e(\xi_k (t) \xi_l (t))&\leq&  \P(S_l > ta(l))\int_0^{t a(l)}  \P(S_k
-S_l \leq ta(k)-x) \d \P_{S_l} (x)\\
&-&\P( S_l \leq t a(l))\int_{ta(l)}^{t a(l_k)}\P(S_k -S_l \leq ta(k)-x) \d \P_{S_l} (x).\end{eqnarray*}
Next for $x\leq t a(l_k)=t \varep(k) a(k)$ we have $\P(S_k -S_l \leq ta(k)(1-\varep(k))) \leq \P(S_k -S_l \leq ta(k)-x)$; moreover for any $x>0$, $\P(S_k -S_l \leq ta(k)-x)\leq \P(S_k -S_l \leq ta(k))$. Thus $\e(\xi_k (t) \xi_l (t))$ is less than or equal to $\P({S_l} \leq ta(l))$ times 
\begin{eqnarray*}
& &\P(S_l>t a(l)) \; \P(S_k -S_l \leq ta(k))- \P(ta(l)< S_l \leq t\varep(k) a(k)) \;  \P(S_k -S_l \leq ta(k)(1-\varep(k)))\\
&\stackrel{def}{=}& \P({S_l} \leq ta(l))\,(x y - x' y')
\end{eqnarray*}
with $x,y,x^{'}, y^{'} \in [0,1]$ such that $x^{'}\leq x$ and $y^{'}\leq y$. Since 
$xy-x^{'} y^{'} \leq (x-x^{'}) + (y-y^{'})$, putting all the previous pieces together gives:
\begin{eqnarray*}
\e(\xi_k (t) \xi_l (t))\leq \P(S_l\leq t a(l)) \biggl(\P(S_l > t\varep(k) a(k)) + \P( t(1-\varep(k))a(k) \leq S_k -S_l \leq ta(k)) \biggr)\\
=\P(S_l\leq t a(l)) \P(S_l > t\varep(k) a(k)) + \P\biggl(S_l \leq ta(l);\;  t(1-\varep(k))a(k) \leq S_k -S_l \leq ta(k)\biggr)
\end{eqnarray*}
\noindent as $S_l$ is independent of $S_k-S_l$. Now since $a(l)\leq a(l_k)=\varep(k) a(k)$, this is
\begin{eqnarray*}
&\leq& \P\biggl(\frac{S_l}{a(l)}\leq t\biggr) \; \P\biggl(\frac{S_l}{a(l)} > t
 \frac{a(l_k)}{a(l)}\biggr) + \P\biggl(t(1-\varep(k)) \leq \frac{S_k}{a(k)}  \leq t(1+\varep (k))\biggr)\\
&\leq& c \;\P\biggl(\frac{S_l}{a(l)}\leq t\biggr) \; 
\biggl(\frac{a(l)}{t a(l_k)}\biggr)^{\alpha^{-}} + \; \P\biggl(t(1-\varep(k)) \leq \frac{S_k}{a(k)}  \leq t(1+\varep (k))\biggr),
\end{eqnarray*}
\noindent 
where we have used Markov's inequality, part $(i)$ of
Lemma \ref{convmoments} (with $\beta=\alpha^{-}$ and $a_k\equiv a(k)$)
to guarantee that  $\e({S_l}/{a(l)})^{\alpha^{-}}\leq c$ for some
constant $c>0$. This  finishes the proof of Lemma \ref{correlations}. 
\end{proof}

\subsection{The proof of (\ref{main})}\label{ss:proof of main}

 Using the previous two lemmas,
this will follow from the result we now state and prove. 

\begin{lem}\label{NtoS}  Under the assumptions of  Theorem
  \ref{t:main2}, there exists  some constant 
$c>0$
such that for $n_0$ large enough, we have
 \begin{equation}\label{letruc}
\limsup_{n\to\infty} \frac{1}{\log n} \sum_{n_0\leq  k \leq n} \frac{1}{k} 
\left(\frac{a (k)}{ S_k}\right)^{\alpha p} \leq c,\;\;\;\; a.s.
\end{equation}
\end{lem}

We begin by proving the lemma then move on to showing how \eqref{main} can be derived from this. 
\begin{proof}[Proof of Lemma \ref{NtoS}]
We write 
\[
P_n \equiv \left(\sum_{n_0 \leq k \leq n} \frac{1}{k}\right)^{-1}\,\sum_{n_0 \leq k \leq n} \frac{1}{k} \delta_{S_k /a(k)},
\]

\noindent 
where $\delta_b$ is Dirac mass at $b$. So proving (\ref{letruc}) is equivalent to showing that 
\begin{equation}\label{principal}
\limsup_{n\to\infty} \int \phi \; \d P_n\; \leq c,\;\;\; a.s.
\end{equation}
where 
\[
\phi (x) = x^{-\beta} (x>0),\;\;\;\; \beta\equiv \alpha p,\;\;\;\; (\alpha<1, \,p>1).
\]

 From the 
a.s.~CLT, see \cite{BerkesDehling93}, page 1643, the sequence of probability measures $(P_n)$ 
converges a.s.~weakly to $Z(1)$. It then follows from page 31 of \cite{Billingsley68} that a.s.~for any {\it bounded} function $\psi$ with  Lebesgue--negligible set of discontinuities, $\int \psi \; \d P_n$ converges to $\e(\psi(Z(1))$, as $n\to\infty$.  

\smallskip

This does not apply to our function $\phi$ as it is not bounded. We circumvent this difficulty by first  expressing $\int \phi\, \d P_n$ as:
\[
\int \phi \; \d P_n =\int_0^{\infty} P_n ((0, \frac{1}{t^{1/\beta}}])
\, \d t=\beta\, \left(\int_0^1+\int_1^{\infty}\right) P_n ((0,x])\, \frac{\d x}{x^{1+\beta}},
\]

\noindent
where we have used a Fubini argument. The integral over $[1,\infty)$
is  bounded so  proving \eqref{principal} is equivalent to checking that 
\[
\limsup_{n\to \infty}\int_0^1 F_n (x) \; \frac{\d x}{x^{1+\beta}} \leq c, \;\;\;\; a.s. 
\]
\noindent 
where
\[
F_n (x) = \frac{1}{\log n} \sum_{n_0 \leq k \leq n} \frac{1}{k} \;
\chi_{\frac{S_k}{a(k)}\leq x} \;\;\biggl( =\frac{1}{\log n}\left(\sum_{n_0 \leq k \leq n} \frac{1}{k}\right)\, P_n((0,x]) \biggr).
\]
Defining $F_{\infty} (x)\equiv\P(Z(1)\leq x)$, we shall prove that a.s.:
\begin{equation}\label{ibragimov}
\limsup_{n\to\infty} \int_0^1 F_n (x) \frac{\d x}{x^{1+\beta}} \leq 2 \int_0^1 F_{\infty} (x) \frac{\d x}{x^{1+\beta}}.
\end{equation}

\smallskip
For simplicity, we set
\[
I_n \equiv  \int_0^1 F_n (x) \frac{\d x}{x^{1+\beta}}\;\; \text{and} \; \; I_{\infty} \equiv  \int_0^1 F_{\infty} (x) \frac{\d x}{x^{1+\beta}};
\]
note that by $(ii)$ of Lemma \ref{convmoments}, $I_{\infty}<\infty.$

\smallskip

Following the proof of Lemma 2 of 
\cite{IbragimovLifshits98}, it is
enough to show that the  $\limsup$ in \eqref{ibragimov} can be taken
along the subsequence $n_m \equiv [e^{{\gamma}^m}]$, for
$\gamma>1$. Indeed, from the definition of $F_n$ the mapping $n\mapsto
\log n \times F_n$ is nondecreasing and hence so is $n\mapsto \log n \times I_n$. Thus for all $n_{m-1} \leq n \leq n_m$:
\[
I_n\leq I_{{n_m}} \; \frac{\log n_m}{\log n}\leq I_{{n_m}} \; \frac{\log n_m}{\log n_{m-1}}
\]

\noindent and the last ratio  is asymptotically equivalent to $\gamma$.

Accordingly,  assuming  that a.s.~$\limsup_{m} I_{n_m} \leq 2 I_{\infty}$, this yields:
\[
\limsup_{n\to \infty} I_n \leq \gamma \limsup_{m\to\infty} I_{n_m}\leq 2\gamma I_{\infty},\;\; a.s.
\]
\noindent 
which delivers \eqref{ibragimov} by letting $\gamma$ decrease to one. 

\smallskip

We now prove that $\limsup_{m} I_{n_m} \leq 2 I_{\infty}$, a.s. By Borel--Cantelli, this would follow from:
\begin{equation}\label{variance}
\limsup_{n\to\infty} \; \e(I_n)\leq I_{\infty}\; \;\text{and}\; \;  \sum_{m} \text{Var}(I_{n_m})<\infty.
\end{equation}
Beginning with $\e(I_n)$, we rewrite it as:
\[
\e(I_n)=\frac{1}{\log n} \sum_{n_0\leq k\leq n} \frac{1}{k} \;
\int_0^1 \, \P(S_k \leq a(k) x)\, \frac{\d x}{x^{1+\beta}}.
\] Using \eqref{l:key}, the above  integral converges to $\int_0^1
\P(Z(1)\leq x) \, x^{-1-\beta}\, \d x$, thus a fortiori its log average does, giving that $
\lim_{n\to \infty} \e(I_n)  = I_{\infty}.$

\medskip

We move on to considering $\text{Var} (I_n)$. By definition, 
\[
\text{Var} (I_n) = \e\left(\int_0^1 (F_n(t) - \e(F_n (t)) \; \frac{\d t}{t^{1+\beta}}\right)^2 = \e\left(\int_0^1 T_n (t) \frac{\d t}{t^{1+\beta}}\right)^2,
\]
\noindent
where 
\begin{equation}\label{T_n}
T_n (t) =  \frac{1}{\log n} \sum_{n_o\leq k\leq n} \frac{\xi_k (t)}{k},
\end{equation}
\noindent with  $\xi_k(t)$ defined as in \eqref{xi_k}. So for any
$\ep>0$ small enough, by Jensen's inequality, we have:
\[
\text{Var}(I_n)=
\frac{1}{\ep^2}\; \e\biggl(\int_0^1 \frac{T_n (t)}{t^{\beta +\ep}}\; {\ep}\; \frac{\d t}{t^{1-\ep}}\biggr)^2\leq \frac{1}{\ep}\; 
\int_0^1 \frac{\e(T_n^2 (t))}{t^{1+2\beta +\ep}} \, \d t.
\]
\noindent 
So as to prove that this last integral is summable along $(n_m)$, we rewrite it as:
\begin{eqnarray*}
& & \int_0^1  \frac{\e(T_n^2 (t))}{t^{1+2\beta +\ep}}  \d t= \int_0^1 \frac{1}{\log^2 n}\left( \sum_{n_0\leq k,l\leq n} \frac{1}{kl}\,\e(\xi_k (t) \xi_l (t)) \right)\; \frac{\d t}{t^{1+2\beta +\ep}}\\
&=&\int_0^1 \frac{1}{\log^2 n}\left( \sum_{n_0\leq k\leq n} \frac{1}{k^2} \, \e(\xi^2_k (t)) \right)\; \frac{\d t}{t^{1+2\beta +\ep}} +2 \int_0^1 \frac{1}{\log^2 n}\left( \sum_{n_0\leq l<k\leq n} \frac{1}{kl} \, \e(\xi_k (t) \xi_l (t)) \right)\; \frac{\d t}{t^{1+2\beta +\ep}}\\
&\stackrel{def}{=}& \text{(I)} +2\; \text{(II)}.
\end{eqnarray*}

\smallskip

Let us begin with $\text{(I)}$. By definition of $\xi_k (t)$, we have that $\e(\xi_k (t)^2)\leq \P(S_k \leq ta(k))$. Accordingly,  
\[
\text{(I)}\leq \frac{1}{\log^2 n} \sum_{n_0\leq k\leq n} \frac{1}{k^2} \int_0^1\P(S_k \leq ta(k)) \frac{\d t}{t^{1+2\beta +\ep}}.
\]
From  \eqref{l:key},  the above integral is  bounded hence so is
the last sum, and therefore  
\[
\text{(I)}= O\left(\frac{1}{\log^2 n}\right),
\]

\noindent 
which is summable  along $(n_m)$ (recall that $n_m = [e^{\gamma^m}]$
with $\gamma>1$).

\medskip
 We next show that the same holds true for $\text{(II)}$. 

\medskip

Setting $l_k\equiv \wh a(\varep(k) a(k))$ with $\varep(t) \equiv
\frac{1}{(\log \log t)^2}$ and $\wh a(\cdot)$ the inverse of $a(.)$,
the sum involved in $\text{(II)}$ splits into two parts according as
to whether $l\leq l_k$ or not; thus,
\[
 (\text{II})=\frac{1}{\log^2 n} \biggl( \sum_{n_0\leq k\leq n }
 \sum_{n_0\leq l\leq l_k} + \sum_{n_0\leq k\leq n} \; \sum_{l_k < l < k} \biggr) \frac{1}{kl} \int_0^1  \e(\xi_k (t) \xi_l (t)) \frac{\d t}{t^{1+2\beta+\ep}}\;\stackrel{def}{=} (\text{III})+(\text{IV}). 
\]

We begin by estimating the covariance $\e(\xi_k (t) \xi_l (t))$ and start with the case where $l_k<l <k.$

\smallskip

Since $\e(\xi_k^2(t))\leq \P(S_k \leq ta(k))$,  the 
Cauchy--Schwarz inequality yields
\[
\e(\xi_k (t) \xi_l (t))\leq \sqrt { \e(\xi_k (t)^2) \e(\xi_l (t)^2)}\leq \sqrt{ 
\P(S_k \leq ta(k))}. 
\]

Hence  
\begin{eqnarray*}
(\text{IV})&\leq& \frac{1}{\log^2 n} \sum_{n_0\leq k\leq n} \frac{1}{k} \biggl(\int_0^1\sqrt{ \P(S_k \leq ta(k))} \frac{\d t}{t^{1+2\beta+\ep}}\biggr)  \sum_{l_k<l< k} \frac{1}{l}\\
&\leq&  \frac{2 }{\log^2 n} \sum_{n_0\leq k\leq n} \frac{1}{k} \; \log\frac{k}{[l_k]} \; \biggl( \int_0^1
\sqrt{\frac{\P(S_k \leq ta(k))}{t^{1+4\beta+2\ep}}} \; \frac{1}{2}\; \frac{\d t}{\sqrt t}\biggr)\\
&\leq& \frac{\sqrt{2 }}{\log^2 n} \sum_{n_0\leq k\leq n} \frac{1}{k} \; \log
\frac{\wh a(a(k))}{[\wh a(\varep(k)a(k))]} \; \sqrt {   \int_0^1
\frac{\P(S_k \leq ta(k))}{t^{1+4\beta+2\ep}} \; 
\frac{\d t}{\sqrt t}}
\end{eqnarray*}

\noindent
by Jensen's inequality. Now using \eqref{l:key},  the integral lying
under the square root above stays  bounded. Next,  since $\wh
a(\cdot)$ is regularly varying of index  $\alpha$,  Potter's Theorem
\eqref{potter} states that picking arbitrary $A>1$ and
$0<\rho<\alpha$, for $x\leq y$ large enough, we have $
\wh a(y)/\wh a(x) \leq A \left(y/x\right)^{\alpha+\rho}.$ Using this
together with an easy computation, by our  choice of $\varep(\cdot)$, we arrive at:
\begin{eqnarray*}
(\text{IV})&=& O\left(\frac{1}{\log^2 n} \sum_{n_0\leq k\leq n} \frac{1}{k} \; \log
\frac{\wh a(a(k))}{\wh a(\varep(k)a(k))}\right)+\, O\left(\frac{1}{\log^2 n}\right)\\
&=& O\left(\frac{1}{\log^2 n} \sum_{n_0\leq k\leq n} \frac{1}{k} \log \log \log k\right)\, +\, O(\frac{1}{\log n})=O\left( \frac{ \log \log \log n}{\log n}\right),
\end{eqnarray*}
which is summable along $(n_m)$. 

\smallskip

Now we turn to $(\text{III})$. Using the estimate of  $\e(\xi_k (t) \xi_l (t))$ for $l\leq l_k$  given in Lemma \ref{correlations}, we write
\[
\text{(III)}\leq \frac{c}{\log^2 n} \sum_{n_0\leq k\leq n} \frac{1}{k{(\varep(k)a(k))}^{\alpha^{-}}} \sum_{n_0\leq l\leq l_k}   \frac{{a(l)}^{\alpha^{-}}}{l} \int_0^1 \P\left(\frac{S_l}{a(l)}\leq t\right) \frac{\d t}{t^{1+2\beta+\ep+{\alpha}^{-}}}
\]
\begin{eqnarray*}
&+& \frac{1}{\log^2 n} \sum_{n_0\leq k\leq n} \frac{1}{k}\int_0^1 \P\biggl(t(1-\varep(k)) \leq \frac{S_k}{a(k)}  \leq t(1+\varep (k))\biggr)\frac{\d t}{t^{1+2\beta+\ep}}   \sum_{n_0\leq l\leq l_k} \frac{1}{l}\\
&\stackrel{def}{=}& \text{(V)}+\text{(VI)}.
\end{eqnarray*}

We start off with $\text{(V)}$.  Since $a(\cdot)$ is regularly varying
of index  $1/\alpha$ then so is $a(l)^{\alpha^{-}}/l$ with index $\frac{\alpha^{-}}{\alpha}-1>-1$. By Karamata's Theorem (\cite{BinghamGoldieTeugels87}, page 26), $n_0$ being large, we have:
\[
\sum_{n_0\leq l\leq l_k} \frac{a(l)^{\alpha^{-}}}{l} \sim \frac{\alpha}{\alpha^{-}} 
a(l_k)^{\alpha^{-}}= \frac{\alpha}{\alpha^{-}} (a(k) \varep(k))^{\alpha^{-}}. 
\]

This together with  \eqref{l:key} yields $
\text{(V)}=O\left( \frac{1}{\log n}\right),$ which is summable along $(n_m)$. 

\medskip

As for $\text{(VI)}$, we write
\begin{eqnarray*}
&\int_0^1&\P\biggl(t(1-\varep(k)) \leq \frac{S_k}{a(k)}  \leq t(1+\varep (k))\biggr)\frac{\d t}{t^{1+2\beta+\ep}}\\
&\leq& \e\left(\int_{\frac{S_k}{(1+\varep(k)) a(k)}}^{\frac{S_k}{(1-\varep(k)) a(k)}}   \frac{\d t}{t^{1+2\beta+\ep}}\right) =O\left( \varep(k) \; \e\biggl(
\biggl(\frac{a(k)}{S_k} \biggr)^
{2\beta+\ep}\biggr)\right)=O(\varep(k)),
\end{eqnarray*}
\noindent 
 by \eqref{lulu}. Now, since $a(\cdot)$ is regularly varying of index 
 $1/\alpha$ and $\varep(\cdot)$ is slowly varying, we have that $\log
 l_k = \log \wh a(\varep (k) a(k)) \sim \alpha \log (\varep(k)
 a(k))\sim \log k$. It follows, by the
 choice of $\varep(\cdot)$ that:
\[
\text{(VI)}=O\left(\frac{1}{\log^2 n} \sum_{n_0\leq k\leq n} \frac{\varep(k) }{k} \log l_k \right)=O\left( \frac{1}{(\log \log n)^2}\right),
\]

\noindent which is summable along $(n_m)$. 

\smallskip

We have finished the proof that 
$\text {Var}(I_n)$ is summable along $(n_m)$; this guarantees
\eqref{variance}, hence \eqref{ibragimov}  and then finally \eqref{principal}. The proof of Lemma \eqref{NtoS} is now complete.
\end{proof}

\subsection{Proving that \eqref{main} follows from Lemma
  \ref{NtoS}}\label{ss:follows from lema}
From \eqref{hNbar}, we have   a fortiori that  $|h({\bbar N}(t))-\wh Z(t)| = o(t^{\alpha})$ a.s.~(log). Hence, a.s.~there exists a set of times ${\cal B}$ of log density zero such that for all $t\notin{\cal B}$, 
$| h({\bbar N}(t))-\wh Z(t)| = o(t^{\alpha})$. As a result,
\[
\lim_{T\to \infty} \frac{1}{\log T} \int_1^T \frac{|h({\bbar N}(t))- \wh Z (t) |}{t^{\alpha}}\; \chi_{t\notin {\cal B}} \frac{\d t}{t}=0,\;\;\; a.s.
\]
\noindent 
where $\chi_A$ is the indicator function of the set $A$. 

\smallskip

So proving that the first limit in (\ref{main}) is a.s.~zero is
equivalent to showing that almost surely,
\begin{equation}\label{main1}
\frac{1}{\log T} \int_1^T \frac{|h({\bbar N}(t))- \wh Z (t)
  |}{t^{\alpha}}\; \chi_{t\in {\cal B}} \frac{\d t}{t}
\longrightarrow 0,\;\;\; \;\;\;T\to\infty.
\end{equation}

From H\"older's inequality, for any $p,q>1$ such that $1/p + 1/q =1$,
the above  is 
\begin{equation}\label{holder}\leq  \left\{\frac{1}{\log T} \int_1^T \left(\frac{|h({\bbar N}(t))- \wh Z (t) |}{t^{\alpha}}\right)^p\; \frac{\d t}{t} \right\}^{1/p} \times \left\{\frac{1}{\log T} \int_1^T  \chi_{t\in {\cal B}} \frac{\d t}{t} \right\}^{1/q}.
\end{equation}
As $T\to\infty$, the last term  approaches the log density of $\cal B$ to the power $1/q$, which is zero. So proving that the first term in \eqref{holder} is eventually bounded would  yield (\ref{main1}). 

\smallskip

Now, since $p>1$, the mapping $x\mapsto |x|^p$ is convex, and so for
all $x,y\geq 0$ we have that $|x-y|^p \leq 2^{p-1} (x^p + y^p)$. Applying this for $x=h(\bbar N(t))$ and $y=\wh Z(t)$, we are done so long as   the log averages of $(\wh Z(t) /t^{\alpha})^p$ and $(h(\bbar N(t))/t^{\alpha})^p$ are bounded.

\smallskip

Since the scaling flow $\wh \tau_t$ for $\wh Z$ is ergodic, using  Birkhoff's ergodic theorem, a.s.
\[
\lim_{T\to \infty} \frac{1}{\log T} \int_1^T \left(\frac{\wh Z (t)}{t^{\alpha}}\right)^p\; \frac{\d t}{t}=\lim_{T\to \infty} \frac{1}{T} \int_0^T ( {\wh \tau}_t (\wh Z)(1))^p \d t =\e(\wh Z(1)^p)<\infty.
\]
So it remains to prove that for some positive and finite constant $c$ we have:
\begin{equation}\label{djoudjou}
\limsup_{T\to\infty}\frac{1}{\log T} \int_1^T \left(\frac{h({\bbar N}(t))}{t^{\alpha}}\right)^p\; \frac{\d t}{t} \leq c,\;\;\; a.s.
\end{equation}
The idea is to decompose the above integral on the partition $([S_{n}, S_{n+1}[)$, $n\geq 0$.  Since $\bbar N$ is a renewal process, 
 for any $S_{n} \leq t < S_{n+1}$ we have that $\bbar N(t)=n$. Thus, it is enough to prove that for $n_0$ large enough, 
\begin{equation}\label{owie}
\limsup_{T\to\infty}  \frac{1}{\log T} \sum_{n_0 \leq n \leq [N(T)]}
h^p (n) \int_{S_{n}}^{S_{n+1}} \frac{\d t}{t^{1+p\alpha}} \leq c\;\;\;\;\;\; a.s. 
\end{equation}

For this, writing $S$ for the polygonal interpolation of $S_n$ and $N\equiv S^{-1}$,  we claim that for all $A$ large enough, a.s.~$N(T)\leq AT$, for $T$ large. Indeed,  we can write
\[
\P(N(T)>AT)= \P(S(AT)<T)\leq \P(S_{[AT]}<T)\leq  e^T\, \left(\e (e^{- X_1})   \right)^{[AT]}
\]

\noindent using Markov's inequality. As $X_1>0$ a.s., the Laplace
transform $\e (e^{- X_1})$ is $<1$, so $\P(N(T)>AT)$ is  exponentially
small for $A$ large enough and we are done by the Borel--Cantelli
Lemma.  Thus we are reduced to proving \eqref{owie} with $[AT]$
replacing $[N(T)]$. But
\begin{eqnarray}
& & \limsup_{T\to\infty}   \frac{1}{\log T} \sum_{n_0 \leq n \leq [AT]}
h^p (n) \int_{S_{n}}^{S_{n+1}} \frac{\d t}{t^{1+p\alpha}}  \label{owie1}\\
 &\leq& \limsup_{T\to\infty} \biggl( \frac{1}{\alpha p \log T} \left(
\sum_{n_0 \leq n \leq [AT]} \frac{h^p (n) - h^p (n-1)}{S_n^{\alpha
    p}} \right)+ \frac{1}{\alpha p \log T}\,\frac{h^p
(n_0)}{S_{n_0}^{\alpha p}}\biggr). \label{2nd}
\end{eqnarray}

The last term of \eqref{2nd} a.s.~converges to $0$ as $T\to
\infty$. And, recalling that $h(\cdot)=a^{\alpha}(\cdot)$ is
regularly varying of index  $1$ and with regularly varying derivative,
Karamata's theorem (see \cite{BinghamGoldieTeugels87} page 12-28)
yields $sh'(s) \sim h(s)$, so for $x$ large we have:
\[
h^p (x+1)-h^p (x) = \int_x^{x+1} (h^p)'(s) \d s = \int_x^{x+1}
p\;\frac{h^p (s)}{s}\;\frac{sh'(s)}{h(s)} \d s  \sim p\,\frac{h^p (x)}{x}.  
\] 
Thus the lim sup in \eqref{owie1}  stays a.s.~bounded by Lemma
\ref{NtoS}, and we have just showed that Lemma \ref{NtoS} implies \eqref{djoudjou}, which in turn delivers \eqref{main1} and thus the first equality in \eqref{main}. 

\smallskip

We note that since $h(\cdot)$, $\bbar N$ and $\wh Z$ are
nondecreasing, it is straightforward to check (following the same pattern as before) that \eqref{main1} holds true with $||h\circ\bbar N-\wh Z||_{[0,t]}$ replacing $|h(\bbar N(t))-\wh Z(t)|$. This shows the second identity in \eqref{main}. The proof of  Theorem \ref{t:main2} is now complete. 

\medskip

We move on to showing how Corollary  \ref{c:ChungErd} follows from Theorem \ref{t:main2}.

\begin{proof}[Proof of Corollary  \ref{c:ChungErd}] As noted in the introduction, since ${\wh \tau}_t$ is ergodic for $\wh Z$, Birkhoff's ergodic theorem gives \eqref{e:logaverageML}. Thus,  Theorem \ref{t:main2}  implies that a.s. 
\[
\lim_{T\to \infty} \frac{1}{\log T} \int_1^T \frac{h(\bbar N(t))}{t^{\alpha}} \frac{\d t}{t} = \e(\wh Z(1)).
\]

This is not quite the statement of Corollary \ref{c:ChungErd}. For this we need to run the reasoning we used in the proof of Theorem \ref{t:main2}, starting with \eqref{hN} instead of \eqref{hNbar}: 
\[
\biggl|\frac{\bbar N(t)}{\wh a(t)} -  \frac{\wh Z(t)}{t^{\alpha}}\biggr | \to 0 \; \;\;a.s. \; (\log);
\]
\noindent 
then, just as we proved that \eqref{main1} followed from
\eqref{djoudjou}, all we have to check is that $\forall p>1$ 
\begin{equation}\label{furst}
 \limsup_{T\to \infty} \frac{1}{\log T} \int_1^T \left(\frac{{\bbar N}(t)}{\wh a(t)}\right)^{p}\; \frac{\d t}{t}\leq c,\;\;\; a.s.
\end{equation}
\noindent 
for some positive constant $c$. As $\wh a(\cdot)$ is regularly varying of index  $\alpha$, by Potter's Theorem \eqref{potter},  
\[
\biggl(\frac{\bbar N(t)}{\wh a(t)} \biggr)^p\leq A\, 
 \max\biggl(\biggl(\frac{h(\bbar N(t))}{t^{\alpha}}\biggr)^{p_{+}},\biggl(\frac{h(\bbar N(t))}{t^{\alpha}}\biggr)^{p_{-}}\biggr)\leq A\, \biggl(\frac{h(\bbar N(t))}{t^{\alpha}}\biggr)^{p_{+}} + A\, \biggl(\frac{h(\bbar N(t))}{t^{\alpha}}\biggr)^{p_{-}} ,
\]
\noindent 
for $A>1$, $\delta>0$ small enough, $t$ large enough, $p_{+}=p(1+\delta)$ and $p_{-}= p(1-\delta)$. 

\smallskip

Choosing $\delta$ small enough so that $p_{-}>1$ then applying
\eqref{djoudjou} for $p=p_{-}$ and $p=p_{+}$ yields \eqref{furst}. As
a result, $\bbar N(t)/\wh a(t)$ and $\wh Z$ share the same log
average,  finishing the proof of Corollary \ref{c:ChungErd}. 
\end{proof}

\section{Cocycles, special flows and the ergodic theorem}\label{s:cocycles}
\subsection{Special flows}\label{ss:specialflows}
First we recall the ergodic theory construction of a {special flow}. 
Given an invertible  measure--preserving 
transformation $T$ of a measure space $(B,\Cal A,\mu_B)$ (referred to
as $(B,\Cal A,\mu_B, T)$ or more simply as $(B,\mu_B, T)$) 
and a measurable function 
$r:\, B\to (0,+\infty)$, we form the space
\begin{equation}
  \label{eq:special flow}
X\equiv \{ (x,t):\, x\in B, \, 0\leq t\leq r(x)\}  
\end{equation}
and make the identification $(x, r(x))\sim (Tx, 0)$.
The flow $\tau_t$ is defined by $\tau_t(x,s)= (x, s+t)$;  a point $(x,0)$ in the {\em base} 
$B= \{(x, 0): x\in B\}$  moves upwards at unit speed, flowing until it reaches the top at  $(x, r(x))$, when it jumps back via the identification to $(Tx, 0)$ in the base and continues.

We write $ \mu$ for the measure induced from the product of $\mu_B$ on the base with Lebesgue measure on $\r$.
This measure is preserved by the flow; one calls  $(X,  \mu, \tau_t)$  the {\em special flow} built {\em over} the {\em base map} $(B,\mu_B,T)$ and 
  {\em under} the {\em return--time function} $r$.
The special flow is conservative ergodic if and only if the base map
is. 
If either the base or the flow space  has finite measure,
conservativity holds (by the Poincar\'e recurrence theorem).
In the special case that $\mu_B$ is a probability measure, the measure
of the flow space is
equal to  the expected return time:
\begin{equation}\label{expreturn}
\mu(X)= \e(r)= \int_B r(x)\d \mu_B (x).
\end{equation}

A basic result is the converse to the construction of a special flow, known as the 
Ambrose--Kakutani Theorem. Given a measure--preserving flow $(X, \A,
\mu, \tau_t)$,  a {\em cross--section} 
of the flow
is a subset $B$ of $X$ such that the orbit of a.e.~point  meets $B$
for a nonempty discrete set of times
(here {\em discrete} means a subset of  $\R$ with no accumulation points). The $\sigma$-algebra of measurable 
sets $\Cal A_B$ of the cross--section    consists by definition of the
subsets $A$  such that the {\em  rectangle} $A_{[a,b]}\equiv \{\tau_t(x):\, x\in A, t\in [a,b]\}$ is $\A$--measurable for  
 $a<b\in \R$; $B$ is said to be  a {\em measurable} cross--section if the
 collection of such rectangles
generates $\A$. The 
{\em cross--section measure} 
$\mu_B$ is defined by $\mu_B(A)= \lim_{r\to 0} \frac{1}{r}\mu(A_{[0,r]})$.

\begin{theo}\label{t:Ambrose}
{\em(Ambrose-Kakutani, \cite{AmbroseKakutani42})} Let $(X,\Cal A, \mu,
\tau_t)$ be a conservative  
ergodic flow. Then there exists a measurable cross--section $B\subseteq X$ with 
finite measure $\mu_B$. The
 flow is isomorphic to the special flow over $B$ with 
return--time function $r(x)= \min \{t>0: \tau_t(x)\in B\}$ and with  
first return map  $(B,\A_B,\mu_B, T)$ where $T(x)= \tau_{r(x)} (x)$. 
\ \ \qed\end{theo} 

\subsection{Cocycles}\label{ss:cocycles}
For several different purposes below we shall need the abstract ergodic theory idea of a 
cocycle over a transformation or flow; we recall the definitions.
\begin{defi} A real--valued {\em cocycle} over a transformation $(X,\Cal A,\mu,T)$ is a measurable function
$\Psi: X\times \Z\to \r $ satisfying for a.e.~$x$
\[
\Psi(x,n+m)= \Psi(x,n)+ \Psi(T^n x,m)\] for all $n,m\in\Z.$ 
A cocycle $\Psi$ over a  flow $(X,\Cal A,\mu, \tau_t)$ is a measurable
function $\Psi: X\times \R \to \R$ which satisfies:
\begin{equation}\label{e:flowcocycle}
\Psi(x,t+s)= \Psi(x,s)+ \Psi(\tau_sx,t)
\end{equation}
 for all $t,s\in\r$. (So in particular, $\Psi(x,0)=0$, for a.e.~$x\in X$).
\end{defi}

A measurable function $\psi:X\to \r$ {\em generates} a cocycle $\Psi$ by summing along the orbits,
defining $\Psi(x,0)= 0$ and 
\begin{equation}\label{e:cocyclesum}
\Psi(x,n)= 
\left\{ \begin{array}{ll} 
&\sum_{i=0}^{n-1}\psi(T^ix) \,\, \text{ for } n>0  \\
-&\sum_{i=n}^{-1}\psi(T^ix) \,\, \text{ for } n<0. 
\end{array} \right.
\end{equation}

For flows the sign is automatically handled by the calculus notation: we define
\begin{equation}\label{e:flowcocyclesum}
\Psi(x,t)= \int_0^t \psi(\tau_s(x))\d s.
\end{equation}
Given a measurable cocycle over a transformation, there is a (unique)
function which generates it: $\psi(x)\equiv \Psi(x,1)$. That this can fail for flows is illustrated by:
\begin{exam}\label{e:Cross--sectionTimes}
Consider a  flow $(X,  \mu,\tau_t)$ with cross--section $B\subseteq X$, and count the number of returns of a point $x$ to the cross--section, by
\begin{equation}\label{TGW}
N_B(x,t)= 
\left\{ \begin{array}{ll} 
&\#\{s\in (0,t]: \tau_s(x)\in B\} \,\,\,\;\;\; \,\text{ for } t\geq 0  \\
-&\#\{s\in (-t,0]: \tau_s(x)\in B\} \,\,\,\, \text{ for } t< 0 
\end{array} \right.
\end{equation}
(so  $\forall x$, $N_B(x,0)= 0$.) Since this 
cocycle is discontinuous in $t$, it cannot
be generated by a function.
\end{exam}

\subsection{Integral of a cocycle}\label{ss:expectedvalue}
 
\begin{defi}\label{d: expectedvalue}
 
The {\em integral} of a cocycle $\Psi $  over a measure--preserving
flow  $ (X,\mu,\tau_t)$ is the extended-real number
\begin{equation}\label{canttakeit}
\I(\Psi)\equiv \frac{1}{t} \int_X \Psi(x,t) \dd\mu(x)\end{equation}
for $t\in \R\setminus\{0\}$ (when that is defined, i.e.~when it is
finite, $+\infty$ or $-\infty$); if this is finite we say the cocycle is {\em integrable}. We say a cocycle $\Psi $  over a measure--preserving flow  $ (X,\mu,\tau_t)$ is {\em nonnegative} a.s.~if for a.e.~$x$, $\Psi(x,t)\geq 0$ for all $t\geq 0$. 
\end{defi}

\begin{prop}\label{p:expectedvalue} Let $ (X,\mu,\tau_t)$ be a
  measure--preserving flow. 
\item{(i)} The integral of a cocycle  \eqref{canttakeit} is indeed independent of $t$. 
\item{(ii)}
If the cocycle $\Psi$ is generated by a function $\psi$, then 
$\I(\Psi)= \int_X\psi(x)\dd \mu(x)$.
\item{(iii)}
 Let $B$ be a flow cross--section with return
map $(B, \mu_B, T)$ and return--time function $r$.
Let $\Psi$ be a measurable cocycle over the flow, which is 
either nonnegative or  integrable.
Then 
\begin{equation}\label{esperance}
\I(\Psi)= \int_B\Psi(x, r(x))\dd\mu_B(x).
\end{equation}
\item{(iv)} If a  flow  $ (X,\mu,\tau_t)$ is conservative ergodic, and
  $\Psi$ is a cocycle with $\I(\Psi)>0,$ then for a.e.~$x$, $\Psi(x,t)\to\pm\infty$
  as $t\to\pm\infty$.
\end{prop}
\begin{proof}[Proof of Proposition of \ref{p:expectedvalue}] We start with $(i)$.
Writing $\lambda(t)=   \int_X \Psi(x,t) \d \mu(x)$, since $\Psi$ is
measurable, for a.e.~$x$, $\lambda$ is a measurable function
on $\r$. We suppose first  that for all $t$, $\lambda(t)$ is finite. 
Now by the cocycle property
together with the fact that $\mu$ is preserved by the flow,
 \[
\begin{aligned}
\lambda(s+t)= \int_X \Psi(x,s+t) \d\mu(x)
= \int_X \Psi(x,t) +
 \Psi(\tau_t(x), s)\d\mu(x)=\lambda(t)+\lambda(s),\\
\end{aligned}
\]
so $\lambda(\cdot)$ is a measurable homomorphism on the additive group of the reals.  

Thus for any $q\in\Q$, $\lambda(qt)= q\lambda(t)$. 
We show that $\lambda $ is in fact $\R$--linear. This type of result goes back at least to Banach
(Theorem 4, Chapter 1 of \cite{Banach32}) in a more general context using Baire measurability; it can be viewed as a rigidity theorem. 
We give a simple direct proof for this case.

Defining $s(t)= \lambda(t)/t$, then $s(qt)= s(t)$ for all 
$t\in\r, q\in\Q$. Setting for $a<b$ 
$$
A_{(a,b)}= \{t: s(t)\in (a,b)\},
$$
then there exist $a,b$ such that this has Lebesgue measure $m(A_{(a,b)})>0$.
By the Lebesgue density theorem, $m$--a.e.~point $t$ of this set is a point of density. Now
for such a $t$, the same is true for $qt$ where $q\neq 0$, since 
Lebesgue measure under dilation by $q$ is multiplied by $q$. We claim that for any 
$A_{(c,d)}$ with $b<c$, then $m(A_{(c,d)})= 0$: if not, taking
 $t_0$ (resp. $t_1$) to be a point of density for $A_{(a,b)}$ (resp. $A_{(c,d)}$), then for any $\varep>0$ there is $q$ such that $qt_0$ is $\ep$-close to $t_1$, giving a contradiction. Taking smaller subintervals, we have that $s(t)$ is $m$--a.s.~constant.  Thus $\lambda$ is a.s.~linear over $\r$.

Let $G\subseteq\r$ denote this set of points $t$ where $s(t)=a_0$ and 
$m(\r\setminus G)= 0$. But in fact $G= \r$; suppose that there
is some $w$ with $s(w)= b_0\neq a_0$. Then for all $x= g+w$ with $g\in G$,
$\lambda(x)= \lambda(g)+\lambda(w)= a_0 g+ b_0 w$ which equals $a_0(g+w)$ 
iff $b_0=a_0$. Hence for all $x\in G+w$, $s(x)\neq a_0$, but this translate of $G$ is also a set of full Lebesgue
measure, giving a contradiction.

Now consider the case where for some $t\neq 0$, $\lambda(t)$ is not finite. 
Define $\wt \lambda(t)= \lambda(t)$ if this is finite, $=0$ otherwise.
Then $\wt \lambda$ is a measurable, finite--valued function and is a homomorphism, so by the 
previous argument is linear, hence is identically zero. Thus $\lambda(t)= +\infty$ or $-\infty$ for each $t\neq 0$. Considering $t>0,$ we define $\wh \lambda(t) = t$  if  $\lambda(t)= +\infty$,  $\wh \lambda(t) = -t$  if  $\lambda(t)= -\infty$. This is still a measurable homomorphism, so as before is linear, so only one choice is possible.

We have shown that if $\forall t\neq 0$ $\lambda(t) $ exists as an extended real number, then 
$\I(\Psi) $ is well--defined.

\smallskip

\noindent
Proof of $(ii)$: Using $(i)$, then Fubini's theorem and the fact that $\mu$ is preserved by $\tau_t$ we have 
\[\I(\Psi)=\frac{1}{t} \int_X\Psi(x,t)\d \mu(x)
= \frac{1}{t} \int_X\int_0^t \psi(\tau_s(x))\d s\, \d \mu(x)=
\frac{1}{t} \int_0^t \int_X\psi(\tau_s(x))\d \mu(x)\, \d s=
 \int_X\psi\d\mu.
\]

\noindent
Proof of $(iii)$: By $(i)$, for 
all $a>0$, 
$$
\I(\Psi)\equiv \frac{1}{a} \int_X \Psi(x,a) \, \d \mu(x).$$
We first claim that if  $\Psi$ is either nonnegative   or   integrable, then
\begin{equation}\label{owen}
\frac{1}{a} \int_X \Psi(x,a) \d \mu(x)=
\frac{1}{a} \int_B\biggl(\int_0^{r(x)} \Psi(\tau_t x, a) \d t\biggr) \d \mu_B (x).
\end{equation}
This is just a Fubini argument,  since
locally $\d \mu= \d \mu_B\times \d t$; precisely,  
consider the measure space $B\times[0,+\infty)$ with 
product measure $\mu_B\times m$ where $m$ is Lebesgue measure, and let
$X$ be as in \eqref{eq:special flow};
we define a function $\wh \Psi$ by $\wh \Psi= \Psi$ on 
$X\subseteq B\times[0,+\infty)$,
zero elsewhere; then we apply Tonelli's or Fubini's theorem (\cite{Royden68} pp.~269-270)
depending on whether  $\Psi$ is nonnegative or integrable, to $\wh\Psi$.

Now by the cocycle property, $\Psi(\tau_t x, a)= \Psi(x,t+a)-\Psi(x,t)$ so
\[
\int_0^{r(x)} \Psi(\tau_t x, a) \d t=
\int_{r(x)}^{r(x)+a}\Psi(x, t) \d t- 
\int_0^{a} \Psi(x, t) \d t.
\]

Since for $x\in B$, $Tx= \tau_{r(x)}(x)$, we have again by the cocycle property,
for any $s\geq 0$,
\[
\Psi(x, r(x)+s) = \Psi(x, r(x))+ \Psi(Tx, s)
.\]
Therefore setting $s=t- r(x)$, 
\[
\int_{r(x)}^{r(x)+a}\Psi(x, t) \d t=
\int_0^a\Psi(x,r(x)+s)  \d s=
\int_0^a\Psi(x, r(x))+ \Psi(Tx, s) \d s=
a\Psi(x, r(x))+ \int_0^a \Psi(Tx, s)\d s.\]
So
\[
\frac{1}{a}\int_0^{r(x)} \Psi(\tau_t x, a) \d t=
\Psi(x, r(x))+ \frac{1}{a}\int_0^a \Psi(Tx, s)\d s-\frac{1}{a}\int_0^{a} \Psi(x, t) \d t
.\]
Inserting the righthand side into \eqref{owen}, since $T$ preserves the measure $\mu_B$, $\int_B\Psi(Tx, s)\d \mu_B(x) 
=\int_B \Psi(x, s)\d \mu_B(x)$ 
and we have \eqref{esperance}.

To prove $(iv)$, by Theorem \ref{t:Ambrose} there exists a
cross--section so we can apply part $(iii)$. If $\I(\Psi)>0$ then using $(iii)$ there exists $c>0$ and
$A\subseteq B$ of positive measure, such that $\Psi(x,
r(x))>c$ for all $x\in A$. Since the flow is recurrent, a.e.~$x$
enters $A$ infinitely often. By the cocycle property plus the fact
that $\Psi$ is nondecreasing, therefore, $\Psi(x,t)\to+\infty$ as $t\to+\infty$.
\end{proof}

We remark that for integrable, nondecreasing $\Phi$, part $(iii)$ can
be proved by an 
approximation argument, taking in \eqref{canttakeit}
a limit as $a\to 0$; the exact ``wraparound'' proof just presented handles  this more general case, where  oscillations and  $\I(\Psi)= +\infty$ are 
allowed.

\subsection{  Ergodic theorems for cocycles over flows.}
Here we prove cocycle versions of the Birkhoff and Hopf ergodic
theorems; these results will be applied in the proof of Theorem \ref{t:logaveragetheorem}.
For transformations, the cocycle theorems are a direct consequence of
 the usual  ergodic theorems as  every cocycle is generated by a
 function. For flows, however, we shall need a different argument.

Let $(X,\mu,\tau_t)$ be a conservative, ergodic, measure--preserving flow.
The standard formulation of the a.s.~ergodic theorems state:
for $\psi:X\to \r$ in $L^1(X, \mu)$,
for $\mu-$a.e.~$x\in X$, 

\noindent
${(1)}$ (Birkhoff Ergodic Theorem)
if $\mu(X)=1$:
\begin{equation}\label{birkhoffth}
\lim_{T\to\infty} \frac{1}{T} \int_0^T \psi(\tau_t x) \ \d t 
= \e(\psi)\equiv \int_X \psi \d\mu.\end{equation} 
\noindent
${(2)}$ (Hopf Ratio Ergodic Theorem) For  $\mu(X)$ finite or infinite, $\psi,\varphi\in L^1(X,\mu)$ with $\int _X\psi\d\mu\neq 0$, then:
\begin{equation}\label{hopfth}
\lim_{T\to\infty} \frac{ \int_0^T \varphi(\tau_t x) \ \d t }
{ \int_0^T \psi(\tau_t x) \ \d t }
=  \frac{\int_X\varphi\d\mu} {\int _X\psi\d \mu}.\end{equation}

\begin{rem}\label{r:HopfLog} For those not so familiar with the Hopf
  theorem we mention that these two 
  theorems are essentially equivalent: for $X$ a probability space,
  taking $\psi\equiv 1$ in \eqref{hopfth} gives
  \eqref{birkhoffth}. Conversely, \eqref{hopfth} can be proved for
  bounded functions from \eqref{birkhoffth} by inducing on 
successively larger sets of finite measure; see \cite{Fisher92}, and
see  \cite{Zweimueller04} for the general $L^1$ case. 

\end{rem}
Let $D=D_{\r}$ denote two-sided Skorokhod path space, the c\`adl\`ag functions
$f:\R\to \R$, let $D_0\subset D$
denote the elements $f$ with $f(0)= 0$, and let
$D_{0+}$ denote  the subset of nondecreasing elements of $D_0$. Then there is a 
bijective correspondence from $D_{0+}$ to the Borel measures
on $\R$, defined by
 \begin{equation}\label{measure to function}
\begin{array}{ll} 
&\rho\bigl((0,t]\bigr)= f(t)\text{ \;\;\; for } t\geq 0  \\
& \rho\bigl((t,0]\bigr)=- f(t)\,\,\text{ for } t< 0,
\end{array}.
\end{equation}
and extending additively from these intervals to the Borel sets;
\eqref{measure to function}  defines, furthermore,  a bijection
between the elements of $D_0$ which are of local bounded variation and the charges
(signed measures) on $\R$.
We recall that 
a function of local bounded  variation can be written
 as $f= g-h$
where $g,h$ are  nondecreasing functions
(Theorem 4,
\S 5.2 of \cite{Royden68}); there is a way to
make this decomposition  canonical:

\begin{defi}Given a function $f\in D_0$ of local
bounded  variation, define a charge $\rho=\rho_f$
by 
\eqref{measure to function}.
Letting $\rho= \rho^+-\rho^-$ be its Hahn decomposition  (p. 236 of \cite{Royden68}), so  $\rho^+,\rho^- $ are
mutually singular measures, we define 
 $f^+, f^-\in D_{0+}$ by reversing  equation \eqref{measure to
   function} for  $\rho^+,\rho^- $.
 Then we call $f= f^+-f^-$ the Hahn decomposition of $f$.
\end{defi}

\begin{lem}\label{l:Hahn decomp} Let  $\Psi$ be a real--valued cocycle over a measurable flow $(X,\mu,
  \tau_t).$ Assume that for a.e.~$x$, 
$\Psi(x,t) $ is 
 a function of local bounded variation  in $t$. Then the 
Hahn decomposition
$\Psi(x,t)= \Psi^+(x,t) -\Psi^-(x,t)$ defines a pair of cocycles $\Psi^+(x,t),\Psi^-(x,t)$.
\end{lem}
\begin{proof}
  The cocycle property implies
$\rho_{\tau_s x}(A)=
  \rho_x(A+s)$; 
$\rho_x^+,\rho_x^-$ satisfy this same equation, and  reversing the logic,
$\Psi^+(x,t)$ and $\Psi^-(x,t)$ defined from these measures 
are cocycles, with $\Psi= \Psi^+-\Psi^-$. 
\end{proof}
Here is our extension of the ergodic theorems to cocycles:
\begin{theo}\label{t:CocycleBirkHopf}
Let  $ \tau_t$ be a conservative, ergodic, measure--preserving flow
on a $\sigma$--finite measure space
$(X,  \mu)$. Let $\Phi,\Psi$ be real--valued cocycles, measurable,  of local bounded variation in $t$.
Assume $\I(\Phi)$, $\I(\Psi)<\infty$ and  $\I(\Phi)\neq 0.$ 
Then, 
for $\mu-$a.e.~$x\in X$,
\smallskip
\item{(i)} (Cocycle Birkhoff  Ergodic Theorem)
if $\mu(X)=1$:
\begin{equation}\label{birkof}
\lim_{T\to\infty} \frac{1}{T} \Psi(x,T)= \I(\Psi).\end{equation}
\item{(ii)} (Cocycle Hopf Ergodic Theorem) For $0<\mu(X)\leq +\infty$:
\begin{equation}\label{hopf}\lim_{T\to\infty} \frac{ \Psi(x,T)}{ \Phi(x,T)}= \frac{\I(\Psi)}{\I(\Phi)}.\end{equation}
\end{theo}
\begin{proof}
 We begin with $(i)$. If $\Psi$ were differentiable along orbits, we could prove this simply by applying the
usual ergodic theorem to the derivative. However as in Example \ref{e:Cross--sectionTimes} above, that may not be the case.
 The idea of the proof will be to, instead, replace the derivative by a difference quotient, and apply the usual Birkhoff 
 Ergodic Theorem for flows  to that function.

Thus, for some fixed $a>0$ we define
$$\Psi_a(x,t)\equiv \frac{1}{a}\bigl(\Psi(x, t+a)-\Psi(x,t)\bigr),
$$
$$\Psi_{-a}(x,t)\equiv \frac{1}{a}\bigl(\Psi(x,t)-\Psi(x, t-a)\bigr)
.$$

Now $\Psi_a(x,0)= \frac{1}{a}\Psi(x,a)$ and $\Psi_{-a}(x,0)= \frac{1}{-a}\Psi(x,-a)$,
so by $(i)$ of Proposition \ref{p:expectedvalue},
\begin{equation}
  \label{eq:integral}
 \int_X\Psi_a(x,0)\d \mu(x)= \I(\Psi)= \int_X\Psi_{-a}(x,0)\d \mu(x),
\end{equation}
$+\infty$ being allowed here. We have therefore,
by the Birkhoff theorem for flows: for a.e.~$x$,
\begin{equation}\label{e:ergthm}
\lim_{T\to\infty}\frac{1}{T}
\int_{0}^{T}\Psi_a(\tau_tx,0)\ \d t=\int_X\Psi_a(x,0)\d\mu(x)= \I(\Psi)
\end{equation}
and similarly for $\Psi_{-a}$.

Now we compare this to $\frac{1}{T}\Psi(x,T)$. For all $s\in\r$, since $\Psi_a(x,t)
=\Psi_a(\tau_t(x),0)$,
we have
\begin{eqnarray}\label{e:cocycleergtheorema}
\frac{1}{a}\int_s^{s+a}\Psi(x,t)\ \d t= 
\int_0^{s}\Psi_a(\tau_t(x),0)\ \d t+
\frac{1}{a}
\int_0^{a}\Psi(x,t)\ \d t \; \; 
 \;{\text{and}}\\
\label{e:cocycleergtheorem-a}
\frac{1}{a}\int_{s-a}^s\Psi(x,t)\ \d t= 
\int_0^{s}\Psi_{-a}(\tau_t(x),0)\ \d t+
\frac{1}{a}
\int_{-a}^0\Psi(x,t)\ \d t.
\end{eqnarray}

Since by Lemma \ref{l:Hahn decomp} 
$\Psi(x,t)$ has a Hahn decomposition, we 
can assume without loss of generality that the cocycle $\Psi$ is nondecreasing. We have, taking $s= T$ successively in  \eqref{e:cocycleergtheorema} and \eqref{e:cocycleergtheorem-a}:
\[
\Psi(x,T)\leq
\frac{1}{a}
\int_{T}^{T+a}\Psi(x,t)\ \d t=
\int_{0}^{T}\Psi_a(\tau_t(x),0)\ \d t
+
\frac{1}{a}
\int_{0}^{a}\Psi(x,t)\ \d t,
\]
$$
\int_{0}^{T}\Psi_{-a}(\tau_t(x),0)\ \d t
+
\frac{1}{a}
\int_{-a}^{0}\Psi(x,t)\ \d t=
\frac{1}{a}
\int_{T-a}^{T}\Psi(x,t)\ \d t\leq \Psi(x,T)
.$$
So
\begin{equation}\label{e:Birkestimate}
\Psi(x,-a)+ \int_{0}^{T}\Psi_{-a}(\tau_t(x),0)\ \d t
\leq 
\Psi(x,T)\leq
\int_{0}^{T}\Psi_a(\tau_t(x),0)\ \d t
+
\Psi(x,a).
\end{equation}

For $x$ and $a$ fixed,  dividing by $T$, and using \eqref{e:ergthm},

$$
\I(\Psi)\leq \lim_{T\to\infty} \frac{1}{T} \Psi(x,T)\leq \I(\Psi)
.$$

\smallskip

Next we prove $(ii)$. Given Hahn decompositions 
$\Psi= \Psi^+-\Psi^-$ and $\Phi= \Phi^+-\Phi^-$, we first prove the theorem under the 
assumption that 
all 
four cocycles have integral $>0$. First we consider the pair 
$\Psi^+, \Phi^+$, assuming to simplify the notation that 
$\Psi=\Psi^+$ and $\Phi= \Phi^+$. From \eqref{e:Birkestimate} we have

\begin{equation}\label{e:Hopfestimate}
\frac{\Psi(x,-a)+ \int_{0}^{T}\Psi_{-a}(\tau_t(x),0)\ \d t}
{\Phi(x,a)+ \int_{0}^{T}\Phi_{a}(\tau_t(x),0)\ \d t}
\leq 
\frac{\Psi(x,T)}{\Phi(x,T)}\leq
\frac{\int_{0}^{T}\Psi_a(\tau_t(x),0)\ \d t+\Psi(x,a)}
{\int_{0}^{T}\Phi_{-a}(\tau_t(x),0)\ \d t+\Phi(x,-a)}.
\end{equation}
From \eqref{hopfth}, and using \eqref{eq:integral},
we have that 
\[\lim_{T\to\infty}\frac{ \int_{0}^{T}\Psi_{-a} (\tau_t(x),0)\ \d t}
{\int_{0}^{T}\Phi_{a}(\tau_t(x),0)\ \d t}= \frac{\I(\Psi)}{\I(\Phi)},
\] and similarly for the upper bound; 
noting that 
$\int_{0}^{T}\Psi_{a}(\tau_t(x),0)\ \d t =
\frac{1}{a}\int_{0}^{T}\Psi(x, a+t)\ \d t \to\infty$ as $T\to\infty$ for a.e.~$x$, by $(iv)$ of Proposition
\ref{p:expectedvalue},
and similarly for $\Psi_{-a}$,
 indeed \eqref{hopf} holds true. 

\smallskip

The remaining pairs $(\Psi^-, \Phi^+)$,
$(\Psi^-, \Phi^-)$, and $(\Psi^+, \Phi^-)$ are handled in the same way.
Next we put this information together.
Now, since all four functions are eventually nonzero, 
we have 
\[
\lim_{T\to\infty}
\frac{\Psi(x,T) }{\Phi^+(x,T)}=\lim_{T\to\infty}
\frac{\Psi^+(x,T) }{\Phi^+(x,T)}- 
\lim_{T\to\infty}
\frac{\Psi^-(x,T)}{\Phi^+(x,T)}=
\frac{\I(\Psi^+)}{\I(\Phi^+)}-\frac{\I(\Psi^-)}{\I(\Phi^+)}= \frac{\I(\Psi)}{\I(\Phi^+)},
\]
which implies \eqref{hopf}  for $(\Psi, \Phi^+)$ hence also for the reciprocal pair $(\Phi^+, \Psi);$
similarly, we have this for the pair $(\Phi^-, \Psi)$ and thus 
\eqref{hopf} holds  for $(\Phi, \Psi) $ and hence for $(\Psi,
\Phi).$  Finally if some of the integrals are zero (but with $\I(\Phi)$ eventually nonzero),
the argument is yet simpler, so we are done.
\end{proof}

We illustrate the cocycle Birkhoff  theorem  with a case where  the usual theorem does not directly apply:

\begin{prop}\label{c:returntimes}
Let $(X,\mu,\tau_t)$ be an ergodic flow on a probability space and 
$B\subseteq X$ a measurable, finite measure cross--section with measure $\mu_B$.  For the cocycle $N_B(x,t)$
of Example \ref{e:Cross--sectionTimes}, then for $\mu_B$--a.e.~$x\in B$, and also for $\mu$--a.e.~$x\in X$,
  $\lim_{T\to\infty} \frac{1}{T} N_B(x,T)= 1/\e(r)$.
\end{prop}
\begin{proof}
  From Theorem \ref{t:CocycleBirkHopf} we know that for $\Psi(x,t)= N_B(x,t)$, then 
$\lim_{T\to\infty} \frac{1}{T} \Psi(x,T)= \I(\Psi)$ for $\mu$--a.e.~$x$. 
Now using $(iii)$ of Proposition \ref{p:expectedvalue},
$\I(\Psi)= \int_B\Psi(x, r(x))\d\mu_B(x) = \int_B1\d\mu_B(x) =\mu_B(B)$. On the other hand,
$\e(r)=  \int_B r(x)\d\frac{\mu_B}{\mu_B(B)}(x)= 1/\mu_B(B)$, finishing the proof.
(The statement for a.e.~$x$ in the cross--section then follows  since ergodic averages are constant along flow orbits).
\end{proof}

\section{The scaling and increment flows and their duals}\label{s:dual flows}
In this section we consider certain 
topological  and Borel measurable flows on two--sided Skorokhod space
 $D= D_\R$  endowed with
 the  $J_1$--topology;  see \cite{FisherTalet11a}
for background and references.

\smallskip

We write
$D_{\geq}\subset D$ for  the nondecreasing functions which go to
$+\infty$ at $+\infty$ and $-\infty$ at $-\infty$, and $D_> $ for the subset of $D_{\geq}$  of  increasing functions.
Recall that  the {\em generalized inverse} of $f\in D$ is
$
\wh f(t)= \inf\{s:\, f(s)>t\}.
$ We write $\Cal I (f)= \wh f$, and  speak of the path $\wh f$ as
being {\em
  dual} to $f$.

Note that $\Cal I:D_{\geq}\to D_{\geq}$. 
Let $D_{0\geq }= D_0\cap D_{\geq}$
and $D_{0>}= D_0\cap D_>$.

We note that:
\begin{lem}\label{l:involution}The map $\Cal I: D\to D$ is a Borel measurable
  involution. It maps $D_{\geq} $ to $D_{\geq}$. The image  $\wh D_{0>}\equiv \Cal I(D_{0>})$ is the collection of
continuous elements of $D_{0\geq}$ such that $0$ is a point of right
increase;  jump points of $f$ become flat spots for 
$\wh f$.  
  \ \ \qed \end{lem}

\subsection{Two pairs of flows}
Choosing $\beta>0$, the {\em scaling flow} $\tau_t$ of index $\beta $ on  $D$ is defined by
\[
(\tau_t f)(x)= \frac{f(e^t x)}{e^{\beta t}}.\] The {\em increment flow} $\eta_t$ is defined by
\[
(\eta_t f)(x)= f(x+t)- f(t)
.\]

\begin{prop}\label{p:firstcommreln}
 The scaling flow  is  $J_1$--continuous,
i.e.~it is a jointly continuous function from $D\times \R$ to $D$;
the increment flow  is (jointly,  Borel) measurable.  The pair of
flows satisfies the following commutation relation: for all
$s,t\in\R$, 
\begin{equation}\label{commutation}
\tau_t\circ \eta_s= \eta_{e^{-t}s}\circ \tau_t.
\end{equation}
\end{prop}
\begin{proof} It is easy to check \eqref{commutation}; for 
the continuity and 
measurability of $\tau_t$ see \cite{FisherTalet11a} and \cite{FisherTalet11b}.
The $\sigma$--algebras generated by the $J_1$-- and uniform--on--compact
topologies on $D$ are the same. Moreover this agrees with the restriction to $D$
of the Kolmogorov $\sigma$--algebra, that generated by the 
product topology on $\r^\r$, i.e.~by the coordinate projections; the
reason is that  
since a path in $D$ has left and right limits, the uniform neighborhood of a path is 
determined by a coordinate neighborhood at a countable dense set of points.
One checks that the increment flow $\eta_t$ on $\r^\r$ is Borel measurable for the 
Kolmogorov $\sigma$--algebra, a fortiori for its restriction to  $D$.
\end{proof}

(Regarding equality of the $\sigma$--algebras and further references  see also 
Theorem 11.5.2 of \cite{Whitt02}.) 
 
\smallskip

Fixing now $\alpha>0$, we let $\tau_t$ denote the scaling flow of index
$1/\alpha$ acting on $D_{0>}$ and  $\wh \tau_t$  the scaling flow of index 
$\alpha$ on $\wh D_{0>}$; $\eta _t$ is the increment flow on $D_{0>}$, while
$\wh\eta_t$ on $\wh D_{0>}$ is defined by duality:
\[
\wh \eta_t \wh f= \wh {(\eta_t  f)}.
\]
The increment flow  on the dual space 
$\wh D_{0>}$ is denoted $\bbar\eta_t$.  We  call $\wh\eta_t$ 
the {\em increment subflow}, as it flows
along a singular subset of times of $\bbar\eta_t$, skipping over
flat intervals of the path $\wh f$.

\begin{prop}\label{p:commreln} 
\noindent 
$(i)$ We have these topological isomorphisms of flows:
\[\wh \tau_{t/\alpha}\circ\Cal I= \Cal I\circ \tau_t \text{ and }
\wh \eta_t\circ\Cal I= \Cal I\circ \eta_t
.\]

\noindent 

$(ii)$ 
$\wh \tau_{ t}$ is continuous, $\wh\eta_t$  is Borel  measurable, and they
satisfy
 the commutation relation
\[
\wh\tau_t\circ \wh\eta_s= \wh\eta_{e^{-\alpha t}s}\circ \wh\tau_t\;\;\;\;\;\;\;\forall s,t.
\]
\end{prop}
\begin{proof}
To prove $(i)$, we write
\[
\wh {\tau_t f} (x) = \inf\{u: \tau_t f(u)>x\}=  \inf\{u: f(e^t\, u)>x\,e^{t/\alpha}\}=\frac{1}{e^t} \wh f( x\, e^{t/\alpha}) = \wh \tau_{t/\alpha} \wh f (x).
\]
By Lemma \ref{l:involution}, $\wh \eta_s $ is
a Borel measurable flow. For the commutation relation of  $\wh\tau_t, \wh \eta_s$ we use $(i)$
 together with \eqref{commutation}, as for all $f\in
\wh D_{0>}$,
\[
\wh\tau_t( \wh\eta_s\wh f)=\wh\tau_t(\wh{\eta_s f})=
\wh{(\tau_{\alpha t}{\eta_s f})}=
\wh{(\eta_{e^{-\alpha t}s}\;\tau_{\alpha t} f)}
= \wh\eta_{e^{-\alpha t}s}(\wh{\tau_{\alpha t}f})
=
 \wh\eta_{e^{-\alpha t}s}(\wh\tau_{ t}\wh f).\]
\end{proof}
 The relationship between $\eta_t$ and
$\bbar \eta_t$ is best expressed in terms of the {\em completed graph}
$\Gamma_f$ of $f$, see also \cite{FisherTalet11b}: 
\[
\Gamma_f=\{(x,y)\in \R\times \R:\, f(x^-)\leq y\leq f(x)\equiv f(x^+)\},
\]
\noindent
where $f(x^-)$ (resp.~$f(x^+)$) stands for the limit from the left
(resp.~right) at $x$.  
Defining the {\em dual graph} $\wh \Gamma_f= \{(y,x):\, (x,y)\in
\Gamma_f\}$, the completed graph of the generalized inverse is
$\Gamma_{\wh f}=\wh{\Gamma}_f.$

\begin{lem}\label{p:dualincrement}
The increment flow $(D_{0>}, \eta_t)$ and the dual increment flow
$(\wh D_{0>}, \bbar \eta_t)$ are related as follows.
For $f\in D_{0>}$ we have:
\[
\wh\Gamma_{\eta_t f}=\Gamma_{\bbar \eta_{f(t)} \wh f}
\]
\end{lem}
\begin{proof} Since $f\in D_{0>}$, every $t$ is a point of right increase of
$f$, so $\wh f( f(t))= t$.  Now,  $\eta_t$ acts on the completed graph via:
\[
\Gamma_{\eta_t f}= \Gamma_f-\bigl(t, f(t)\bigr), \;\;\;\; {\text{so}}
\]
\[
\wh\Gamma_{\eta_t f}= \wh{\bigl(\Gamma_f-(t,
  f(t)\bigr)}=\wh\Gamma_f-(f(t),t)= \Gamma_{\wh f}-\bigl(f(t),\wh
f(f(t))\bigr)=\Gamma_{\bbar \eta_{f(t)} \wh f}.\] \end{proof}
\begin{cor}\label{c:returntime}
Let $B\subseteq D_{0>}$ satisfies that  
the return--time function $r(f) = \min\{t>0:\, \eta_t(f)\in B\}$ exists
for all $f\in B$ (and so is everywhere
finite and strictly positive). 
 Then $\wh B$ satisfies this for the increment subflow  $\wh\eta_t$ with
return--time function $\wh r(\wh f)= r(f)$, and also for the dual 
increment flow $\bbar \eta_t$, with return--time function
$\bbar r(\wh f)= f(r(f))= f(\wh r(\wh f))$.
\end{cor}
\begin{proof} For $\wh \eta_t$ this follows from the isomorphism
  with $\eta_t$. Indeed,
\[
\wh r(\wh f)= \min\{ t>0: \wh \eta_t \wh f \in \wh B\}= \min\{ t>0: \wh {\eta_t f} \in \wh B\}=r(f).
\]

For $\bbar\eta_t$, say
 for some $f\in B$,  $a= r(f)$. 
Then
$\eta_{a} f= g\in B$. From Lemma \ref{p:dualincrement}, 
\[
\wh
\Gamma_{ \eta_{a}  f}= \wh \Gamma_{g}= \Gamma_{\bbar \eta_{b} \wh f}\] for 
$b= f(a)= f(r(f))$. This is the least such time and so is the return time for the flow $\bbar \eta_t$.
\end{proof}

\section{Measures for the  increment flow}\label{s:measuredualincrement}
We now study invariant measures for the increment flow, beginning with
the dual flow of \S \ref{s:dual flows}.
The starting point will be the general framework of ergodic self--similar processes with  stationary increments.
Then in \S \ref{ss:incrementML} we 
specialize to completely asymmetric stable processes of index $\alpha\in (0,1)$ 
and their duals, the  Mittag--Leffler processes. In
\S\ref{ss:renewalflow} we treat the renewal flow, and in \S
\ref{ss:integergaps} we examine the special case of integer gaps.

\begin{defi}
By an {\em ergodic self--similar process $X(t)$ of index $\beta>0$} and with paths in $D$ 
we mean that we have 
a probability measure $\nu$ on $D= D_\R$ (resp.~$D_{\R^+}$)  a {\em two--sided} (resp.~{\em
  one--sided}) process which is 
invariant and ergodic for the scaling flow of index 
$\beta$.
We say the process has {\em stationary increments} iff this same 
measure is invariant for the increment flow $\{\eta_t\}_{t\in\r}$
(resp.~{\em semi}flow $\{\eta_t\}_{t\geq 0}$) in the case of $D_\R$ (resp.~$D_{\R^+}$).
\end{defi}

\begin{rem} 
The definitions of self--similar process, or a process with stationary increments, 
can equivalently be formulated for finite cylinder sets since these
uniquely determine the measure on $D$. A one--sided process which is $\{\eta_t\}_{t\geq 0}$--invariant
has a unique extension to 
a two--sided $\{\eta_t\}_{t\in\r}$--invariant process: one first extends via $\eta_T$--invariance for fixed $T<0$ 
to the Borel $\sigma$--algebra on $D_\r$ over times in the interval $[T, \infty)$, then notes that the 
nested union as $T\to -\infty$ of these generates the full $\sigma$--algebra of $D_{\r}$. If the one--sided
process is self--similar, then so is  the two--sided process. Note that self--similarity forces $X(0)= 0$ a.s.
\end{rem}
Now assume we are given a two--sided process which is self--similar with stationary increments, 
with paths in $D_{0>}$; we denote by $\nu$ the
probability measure on
path space  $D_{0>}$, which is 
invariant and  ergodic  for the flows $\tau_t$ and $\eta_t$. 
Via Proposition \ref{p:commreln} the pushed--forward measure $\wh \nu= \Cal I^*(\nu)$ on the dual space $\wh D_{0>}$ is  invariant for  $\wh\tau_t$ 
and for the increment subflow $\wh\eta_t$.

 We next define from this  a measure $\bbar\nu$
 which will be invariant for the dual increment flow
$\bbar\eta_t$ itself, and which may be infinite.

\smallskip

To carry this out, we first define functions $\Psi: D_{0>}\times \R\to \R$ and 
$\Phi, \wt\Phi, \wh\Psi: \wh D_{0>}\times \R\to \R$ 
by
\begin{eqnarray*}
\Psi(f,t)= f(t),\;\;\;\;& &\;\; \Phi(\wh f,t)= \wh f(t),\\
\wt \Phi(\wh f, t)= \wh f(f(t)),\;\;\;\;& &\;\; \wh\Psi(\wh f,t)= f(t).\end{eqnarray*}

\begin{prop}\label{p:is a cocycle} The functions $\Psi$ and  $\Phi$ are cocycles over the flows $\eta_t, \bbar \eta_t$ respectively, and  $\wt \Phi$ and $\wh \Psi$
are cocycles over $ \wh\eta_t$.
\end{prop}
\begin{proof}
That $\Phi$, $\Psi$ and $\wh \Psi$ are cocycles follows from the definitions. We only say a word about $\wt \Phi$: since for  $f\in D_{0>}$  every $t$ is a point of increase,
 $\wt\Phi(\wh f,t)= t$.  And so 
\[
\wt\Phi(\wh f,t+s) =t+s= \wt\Phi(\wh f,t)+ \wt \Phi(\wh{\eta_t f},s)
=
\wt\Phi(\wh f,t)+ \wt \Phi(\wh\eta_t \wh f,s)
\]
so this is a cocycle over $\wh \eta_t$.
\end{proof}

Next we define the promised invariant measure for $\bbar\eta_t$.

First, since the flow $\eta_t$ preserves the probability  measure $\nu$, it is conservative
(by the Poincar\'e recurrence theorem); it is futhermore assumed ergodic, 
 and so 
 we can apply the 
  Ambrose--Kakutani theorem,  choosing  a finite measure cross--section $(B,\nu_B,T)$, 
with first return map $T$ and return--time function $r:\, B\to  (0,+\infty)$.
Setting $\wh T(\wh f)\equiv \wh{(T f)}$ and $\wh \nu_{\wh B}= \Cal I^*(\nu_B)$, 
by the isomorphism of $\eta_t$ with $\wh\eta_t$ we have that
$(\wh B,\wh\nu_B,\wh T)$ is a cross--section for the flow 
$(\wh D_{0>}, \wh\eta_t, \wh \nu)$,  with return time $\wh r(\wh f)= r(f)$.
We then form a second special flow over 
$(\wh B,\wh\nu_B,\wh T)$, choosing a different return--time function.

\begin{prop}\label{p:specialflowetahat}
The special flow over $(\wh B,\wh\nu_B,\wh T)$  with return time $\bbar r(\wh f)= 
\wh \Psi(\wh f,\wh r(\wh f))
$
is a special flow representation for 
$\bbar\eta_t$; indeed, for $\wh\nu_B$--a.e.~$\wh f\in\wh B$,
the return time for $\bbar\eta_t$ to 
$\wh B$ is $\bbar r(\wh f)$.
\end{prop}
\begin{proof}
But this is now immediate from Corollary \ref{c:returntime}.
\end{proof}

\noindent
We transport the special flow measure  to a measure written $\bbar
\nu$  on the dual increment flow
$(\wh D_{0>},\bbar\eta_t)$. The next step will be to calculate its
total mass; for that we first show:

\begin{prop}\label{p:expectedvalofcocycles}
The integrals of the cocycles 
are $\I(\Psi)=\I(\wh\Psi)=\int_{D_{0>}} f(1)\, \dd\nu(f)$ and $\I(\Phi)=\I(\wt\Phi)=1$.
\end{prop}
\begin{proof}
The integral of the cocycle $\Psi$ is, by $(i)$ of Proposition \ref{p:expectedvalue},  \[
\I(\Psi)=\int_{D_{0>}} \Psi(f,1)\,\d \nu(f)= \int_{D_{0>}} f(1)\,\d \nu(f).\]
Similarly, 
\[
\I(\wh\Psi)= \int_{\wh D_{0>}} \wh\Psi(\wh f,1)\,\d \wh \nu(\wh f)= \int_{\wh D_{0>}} f(1)\,\d \wh \nu(\wh f)= \int_{D_{0>}} f(1)\,\d \nu(f).\]

To evaluate the integral of $\Phi$, we choose, as in Proposition \ref{p:specialflowetahat}, a finite measure cross--section $\wh B$ for $\wh \eta_t$
with return--time function $\wh r$. This is also a cross--section for $\bbar \eta_t$, with return time $\bbar r(\wh f)$ and with equal base measure $\bbar \nu_{\wh B}=\wh \nu_{\wh B}$.
We have, using Proposition \ref{p:expectedvalue} $(iii)$,
\[
\I(\Phi)= \int_{\wh D_{0>}} \Phi(\wh f,1)\,\d \bbar \nu(\wh f)=
\int_{\wh B} \Phi(\wh f,\bbar r(\wh  f))\,\d \bbar \nu_{\wh B}(\wh f)=
\int_{\wh B} \wh f(\bbar r(\wh  f))\,\d \bbar \nu_{\wh B}(\wh f)=
\int_{\wh B} \wh f(\bbar r(\wh  f)) \,\d \wh \nu_{\wh B}(\wh f).\]
Now by Corollary \ref{c:returntime}, $\bbar r(\wh  f)= f(\wh r(\wh f))$,
so $\wh f(\bbar r(\wh  f))= \wh f(f(\wh r(\wh f)))=\wh r(\wh f)$ and  this is
\[
\int_{\wh B}\wh r(\wh f) \,\d \wh \nu_{\wh B}(\wh f)
=\wh\nu(\wh D_{0>})=1.
\]
And lastly, 
\[
\I(\wt \Phi)= \int_{\wh B} \wt \Phi(\wh f, \wh r(\wh f))\,\d \wh \nu_{\wh B}(\wh f)=
\int_{\wh B}\wh r(\wh f) \,\d \wh \nu_{\wh B}(\wh f)
=1.
\]
\end{proof}

\begin{cor} \label{c:totalmass}
The dual increment flow
 $(\wh D_{0>},\bbar\nu,\bbar\eta_t) $ is conservative, and it is ergodic iff 
the increment flow $(D_{0>},\nu,\eta_t)$ is;
the total mass for the invariant measure $\bbar \nu$ is
$\bbar\nu(\wh D_{0>})= \int_{D_{0>}} f(1)\, \dd\nu(f)$.
\end{cor}
\begin{proof}
Again,   being conservative is automatic   when the 
space has finite measure; hence the cross--section map is recurrent. 
Ergodicity and recurrence  then pass from the cross--section map to the flow.

The total mass of the special flow is 
\[
\int_{\wh B} \bbar r(\wh f) \,\d \wh \nu_{\wh B}(\wh f)= 
\int_{\wh B} \wh\Phi(\wh f, \wh r(\wh f) )\,\d \wh \nu_{\wh B}(\wh f)= 
\I(\wh \Phi)= \I(\Psi)= \int_{D_{0>}} f(1)\, \d\nu(f)
\]
by $(iii)$ of Proposition \ref{p:expectedvalue} together with
Proposition \ref{p:expectedvalofcocycles}.
\end{proof}

\subsection{A measure for the increment flow of the Mittag--Leffler process}\label{ss:incrementML}
Here $Z$ denotes a positive and completely asymmetric two--sided
$\alpha$--stable process ($\alpha\in (0,1)$); its law $\nu$ lives on
 $D_{0>}$ and is invariant for $\tau_t$   the scaling flow of index
 $1/\alpha$. Its generalized inverse $\wh Z$, with
law $\wh \nu$ invariant for the scaling flow $\wh\tau_t$ of index
$\alpha$, lives on 
$\wh D_{0>}$, while the measure $\bbar \nu$ lives on $D_{0\geq}$.   We
write  $\bbar \eta_t$, $\wh \eta_t$ for the increment flow and subflow respectively.

\begin{prop}\label{p:nextcommreln}
The pairs of flows $\tau_t, \eta_t$ on $(D_{0>}, \nu)$, and 
$\wh \tau_t, \wh \eta_t$ on $(\wh D_{0>}, \wh \nu)$ are Bernoulli
flows of infinite entropy. They satisfy the  commutation relations and
properties in Propositions  \ref{p:firstcommreln} and
\ref{p:commreln}. 
The pair $\wh \tau_{ t},\bbar\eta_t$ satisfies the commutation
relation of Proposition  \ref{p:firstcommreln}; 
$(D_{0\geq},\bbar\nu,\bbar\eta_t)$ is a
conservative ergodic flow,
with $\bbar \nu$ a $\sigma$--finite, infinite measure.
\end{prop}
\begin{proof}
That $\tau_t$ is a Bernoulli flow of infinite entropy on the
Lebesgue space $(D,\nu)$
is proved in Lemma 3.3 of \cite{FisherTalet11a}.
By Proposition \ref{p:commreln}, the flows $(D, \wh\nu,\wh \tau_{t/\alpha })$ and $(D, \nu,\tau_t)$ are isomorphic, so the same holds for $\wh \tau_t$. 

Next we consider the increment flows.
Since the stable process $Z$ has stationary  increments, $\eta_t$ preserves $\nu$. Since these increments are, 
moreover,  independent, $(X_n)_{n\in \Z}$ with 
$X_n\equiv Z(n+1)-Z(n)$ is an i.i.d.~sequence, with law $G_{\alpha,1}$.
This is an infinite entropy Bernoulli shift; its path space is $\Pi= \Pi_{-\infty}^{+\infty} \r$
with infinite product measure $\mu_\alpha$ of distribution $G_{\alpha,1}$ on each coordinate and left shift map $\sigma$.
Now $X_n= (\eta_n(Z))(1)$, so $Z\mapsto (X_n)_{n\in \Z}$ defines a 
measure--preserving map  from $(D, \nu, \eta_1)$ to $(\Pi, \mu_\alpha, \sigma)$.
That is, the time--one map $\eta_1$ has a homomorphic image which is a Bernoulli 
shift of infinite entropy. The same is true for each time $1/k$ map for 
$n=1,2,\dots$.  These increasing $\sigma$--algebras of Borel sets separate 
points of $D$, hence generate the Borel $\sigma$--algebra, so $\eta_1$ is itself indeed a Bernoulli shift and 
$(D,\nu,\eta_t)$ is a Bernoulli flow.

Now for $\alpha\in(0,1)$ and $Z$ positive and completely asymmetric by Propositions \ref{p:commreln} and 
\ref{p:nextcommreln}
the flow $(D, \nu, \eta_t)$ is isomorphic to 
$(D, \wh\nu, \wh\eta_t)$, hence this second flow
 is also  infinite entropy Bernoulli. 

As we saw above, the fact that the flows $(D, \wh\nu, \wh\eta_t)$ and $(D, \bbar\nu, \bbar\eta_t)$  
share a common cross--section 
allows the ergodicity of  $\wh\eta_t$ to pass over to recurrence and ergodicity for the infinite measure flow $\bbar\eta_t$.
Lastly, 
from Corollary \ref{c:totalmass}, the total mass is
$\bbar\nu(\wh D_{0>})= \e(Z(1))= +\infty$, as $0<\alpha<1$.
\end{proof}

\subsection{A measure for the renewal flow}\label{ss:renewalflow}

 We fix a probability measure $\mu_F$ on $ \r^+$   
with distribution function $F$ and place the infinite product measure
 $\mu= \otimes_{-\infty}^\infty\mu_F$ on $B\equiv \Pi_{-\infty}^\infty
 \r^+$. The left shift map $\sigma$ acts on the Lebesgue space $(B, \mu)$, 
with points $x= (\dots x_{-1}x_0 x_1\dots)$. Then 
$(B, \mu,\sigma)$ is  a Bernoulli shift (in the generalized sense: the
``state
space'' is  the support of $\mu_F$  and may be uncountable).
We define the {\em renewal flow} $(X, \rho,h_t )$ to be  the special flow 
over $(B, \mu, \sigma )$ with return--time function $r(\xx)=
x_0$, see \S \ref{ss:specialflows}.

Next we  represent
the renewal flow 
 as an increment flow on  the  nondecreasing paths
$D_{0\geq}$; this will make
 precise the relationship between the renewal flow and renewal
process. For this  
we define a map 
$\zeta: X\to  D_{0\geq}$ by
$\zeta(x)\equiv N_B(x,\cdot)$, where $N_B$ counts the number of
returns of the flow $h_s$ to
the cross--section as in Example \ref{e:Cross--sectionTimes}, equation \eqref{TGW}.
We write $\bbar \mu$ for the pushed--forward measure $\zeta^*(\rho),$ and $\wh \mu$ for the measure 
pushed forward from $\mu$ on the cross--section $B$.
We denote by $\bbar\eta_t$  the increment flow on $D_{0\geq}$  (extending the previous definition 
from $\wh D_{0>}\subseteq D_{0\geq}$). 

We have:
\begin{prop}\label{p:renewal flow} 
The renewal flow $(X,\rho,h_t)$ is
  conservative ergodic.
The map $\zeta$ gives a flow  isomorphism from 
the renewal flow $(X,\rho,h_t)$ to the increment flow $(D_{0\geq},\bbar \mu,\bbar\eta_t)$. The total mass of the  flow is
$\rho(X)=\int_{\R^+} s\, \dd F(s)$.
\end{prop}
\begin{proof} Since the base map is conservative ergodic (it is a
  Bernoulli shift), so is the flow.
The map $\zeta $ is a bijection; that it is   a conjugacy, i.e. $\zeta\circ h_t= \bbar\eta_t\circ \zeta$, follows from the definition of the increment flow.  From \eqref{expreturn}, the total mass for the flow $(X,\rho,h_t)$ is $\int_B x_0\, \d \mu(\xx)=\int_{\R^+} s\, \d F(s)
.$\end{proof}
Note that we have defined two different invariant measures for the
same flow, the 
increment flow $\bbar\eta_t$  on $ D_{0\geq}$: 
$\bbar \nu$ for the two--sided Mittag--Leffler process and $\bbar \mu$ for the  renewal process.

\subsection{The case of integer gaps}\label{ss:integergaps}
For the  special case of a renewal
process with integer gaps
of distribution function $F$, we describe in this section
an alternate construction of the renewal flow as
the {\em suspension flow} over a map called the {\em renewal
transformation} defined by $F$, that is,  the special flow over that transformation with constant
return time one.
We begin by presenting 
 several equivalent
models for this map.

The gap
distribution  $\mu_F $ lives  on $\N^*\subseteq \R^+$. For a
first model of our map, we define 
$\mu=\otimes_{-\infty}^\infty\mu_F$ on $\Sigma=
\Pi_{-\infty}^{+\infty}\N^*$, so    $(\Sigma, \mu,\sigma)$ is a
Bernoulli shift with this countable alphabet. 
With $\xx= (\dots x_{-1}x_0 x_1, \dots)\in \Sigma$, and $B_k=\{\xx: x_0= k\},$ 
we  define the 
 tower transformation 
 $(\wh\Sigma, \wh\mu, T)$ with base $(\Sigma, \sigma, \mu)$ and  return time $k$ 
 over $B_k$. Thus,
 $\wh\Sigma= \{(\xx, j): \, 0\leq j\leq k-1\}\subseteq \Sigma\times\N$ with map $T((\xx, j))= (\xx, j+1)$ if 
$j<k-1$ and $T((\xx, j))= (\sigma(\xx), 0)$ if $j=k-1$. The tower
measure $\wh\mu$ is the restriction of the product of $\mu$ with
counting measure on $\N$ to $\wh\Sigma$; note that the
$\wh\mu$--measure of the tower base is $1$.
We call $(\wh\Sigma, \wh\mu, T)$ the {\em tower
  model} of the renewal transformation.

For a second  model which we term the {\em event process}, 
 let $\psi= \chi_B$ be the indicator
function of the tower base $B=\Sigma$, and given a point $(\xx, k)$ in
$\wh\Sigma$,    let $(Y_i)_{i\in \Z}\in \Pi_Y\equiv \Pi_{-\infty}^{+\infty}\{0,1\}$
be the sequence $Y_i=Y_i(x,k)= \chi_B\circ T^i((\xx, k))$. This defines a function from
$\wh\Sigma$ to $\Pi_Y$ which conjugates $T$ to the shift map $\sigma$.
We write $\mu_Y$ for the 
push-forward of $\wh\mu$, and  note that the $\mu_Y$--measure of $[Y_0=1]$  is $1$.

For a third model,
we define $\Psi((\xx,k),n) $
to be the cocycle over the tower model generated by $\psi= \chi_B$ (see \eqref{e:cocyclesum}). 
Writing $\wt N_n(\xx,k)\equiv \Psi((\xx,k),n)$, 
let $\wh\mu_{\wt N}$ denote the pushed--forward measure 
on $\Pi_{\wt N}\equiv\Pi_{-\infty}^{+\infty}\Z$, via the map which sends $(\xx,
k)$ to $\wt N_n=\wt N_n(\xx,k) $. Then
the {\em increment shift} $(\Theta(\wt N))_n= \wt N_{n+1}- \wt N_1$
preserves $\wh\mu_{\wt N}$.

The fourth model is
 a countable state Markov shift,  the {\em
  renewal shift}. The state space is $\N^*$; for the   probability vector 
$p=
(p_k)_{k\geq 1}$ on $\N^*$ with  $p_k=\mu_F (k)$,
we define the
 transition matrix 
\[P=\left[\begin{matrix}
p_1& p_2&p_3&p_4\dots\\
1 &  0&0&0\dots\\
0 &  1&0&0\dots\\
0 &  0&1&0\dots\\
\vdots& & & \\
\end{matrix}\right]\]

The shift--invariant Markov measure $\wh\mu_P$ on $\Pi =\Pi_{-\infty}^{+\infty} \N^*$ is then defined in the
usual way from the transition matrix plus an invariant row vector
$\pi\equiv(\pi_1,\pi_2,\dots)$.  As $\pi P=\pi$,
 choosing  $\pi_1=1$ 
determines this vector uniquely: $\pi_k= \sum_{i\geq k}^\infty p_i.$

We have:
\begin{prop}

These four models of renewal transformation with integer gap
distribution $F$ are measure
theoretically and topologically isomorphic:

\item{(i)} the  tower $(\wh\Sigma, \wh\mu, T)$ over the
  Bernoulli shift 
$(\Sigma, \mu,\sigma)$, with heights equal to the gaps;
\item{(ii)} 
the left shift map on the event process, $(\Pi_Y, \mu_Y,\sigma)$;
\item{(iii)} the increment shift $\Theta$ on  $(\Pi_{\wt N},
  \wh\mu_{\wt N})$;
\item{(iv)} 
the renewal shift $(\Pi,  \wh\mu_P,\sigma)$ with transition matrix $P$ where  $p_k=\mu_F(k)$,
and with  initial vector 
$\pi\equiv(\pi_1,\pi_2,\dots)$ given above.
The total mass of the renewal transformation equals 
$$\sum_{k=1}^\infty \pi_k= \sum_{k=1}^\infty \sum_{i=k}^\infty p_i=
\sum_{k=1}^\infty kp_k= \int_0^{\infty}s \, \dd F(s)
.$$

The renewal flow is measure theoretically and topologically isomorphic to the
suspension flow over the renewal transformation,
and its  total mass  is that of the transformation.
\end{prop}
\begin{proof}

The renewal flow was
constructed above as a special flow over the Bernoulli  shift 
on $\Pi_{-\infty}^{+\infty}\R^+$ with independent product measure 
of the distribution $F$ and with return--time
function 
$r(\xx)=
x_0=k$. Since $F$ is supported on the integers, we can replace
$\Pi_{-\infty}^{+\infty}\R^+$ by
 $\Pi_{-\infty}^{+\infty}\N^*$, which is the  Bernoulli  shift space
 for the base of the tower model. Since the return time for the tower is the
same as for the flow, 
the flow is the supension flow of this map.
Now the mass of the suspension flow is on the one hand given in Proposition
\ref{p:renewal flow}, as
$\int_0^{\infty}s \, \d F(s)
$; on the other hand it  equals that of its base map (the tower model).

The renewal transformation $(\wh\Sigma, \wh\mu, T)$ is naturally isomorphic to 
the event process shift $(\Pi_Y, \mu_Y, \sigma)$, since the
correspondence 
$(\xx, k)\mapsto Y_i(\xx,k)$ is one--to-one, as cylinder sets of the event process
determine the 
return--time partition of the base which in turn defines the Bernoulli shift.

 For each $(\xx,k)$,
$\wt N_0=0$ and 
$\wt N_n= \Sigma_{i=0}^{n-1} Y_i$ for $n>0$,
$\wt N_n= \Sigma_{i=n}^{-1} Y_i$ for $n<0$.
This defines a map from  $\Pi_Y$ to $\Pi_{\wt N}$, which conjugates the 
 shift to the
increment shift and which
is a bijection as 
$ Y_n= \wt N_{n+1}- \wt N_n$ gives the inverse.

We define a map from the  renewal shift to the tower map, sending 
first the set $\wt B_k\equiv\{ x\in
\Pi:\, x_0= 1 \text{ and } x_1=k+1\}$ to $B_k\times \{0\}$; these  have the same
measure, $p_k$. Then $\{x: x_0=1\}= \cup_{k\geq 1} \wt B_k$ maps to the base
$\cup_{k\geq 1} B_k$, and both $\wt B_k$ and $B_k$ have the same
return time $k$  to $\wt B$,  $B$, hence the dynamics is conjugated.   The independence of the gaps (the Markov
property) is registered in the
Bernoulli property of the base map, so the measures correspond.
Choice of $\pi_1=1$ corresponds to the tower base having measure $1$.

We mention that the correspondence from the renewal shift to the event
process sends, e.g., the sequence $(\dots 154321321\dots)$ to $(\dots
100001001\dots)$, also clearly a bijection.

Now we have seen in the first paragraph that the 
mass of the tower model equals $\int_0^{\infty}s \, \d F(s)
$. On the other hand, it can be calculated 
by  adding up the 
measure of the column over each $B_k$, giving
$\wh\mu(\wh\Sigma)= \sum_{k=1}^\infty k\mu(B_k)= \sum_{k=1}^\infty
kp_k$. Calculated for the renewal shift,
the total mass is that of
its invariant vector: $\sum_{k=1}^\infty \pi_k= \sum_{k=1}^\infty \sum_{i=k}^\infty p_i=
\sum_{k=1}^\infty kp_k$. 
\end{proof}

\section{Order--two ergodic theorems }\label{s:log average ergodic}

\subsection{Order--two ergodic theorems for self--similar processes}\label{ss:logaverageErgodicThms}

We begin in the context of  a self--similar process, dual to a self--similar process with stationary increments.

\begin{theo}\label{t:logaverageself--sim}
For $\alpha>0$, let  $\nu$ be a probability measure on $ D_{0>}$, invariant and ergodic for the scaling flow $\tau_t$ of index $1/\alpha$ and also for the increment flow $\eta_t$. Let 
$\wh \tau_t$ denote the dual scaling flow of index $\alpha$ on $(\wh D_{0>}, \wh\nu)$, and 
$(\wh D_{0>},\wh \nu, \wh \eta_t)$  the increment subflow. We write
$\bbar\nu$ for the $\bbar \eta_t$- invariant measure  defined in the
first part of \S \ref{s:measuredualincrement}.
Let  $\Phi(\wh f,t) $
be a cocycle for the increment flow $\bbar\eta_t$which  is jointly
measurable, of local bounded variation in $t$, with $\I(\Phi)$
finite. Then for $\bbar \nu-$a.e.~$\wh f$
we have:
\begin{equation}\label{first}
\lim_{T\to\infty} \frac{1}{\log T} \int_1^T \frac{\Phi(\wh f,t)}{t^\alpha} \frac{\dd t}{t}=  c_0\;\I(\Phi),\end{equation}
with  $c_0\equiv  \int_{\wh D_{0>}}\wh f(1)\dd\wh\nu(\wh f)$. 
\end{theo}
\begin{proof}
By Corollary
\ref{c:totalmass}, since $\wh\eta_t$ is ergodic, 
$(\wh D_{0>},\bbar \nu, \bbar \eta_t)$ is conservative ergodic, with
total  mass $\bbar \nu(\wh D_{0>} )= \int_{D_{0>}} f(1) \d\nu(f)$.
Hence from the Hopf Theorem \eqref{hopf}, it is enough to 
prove \eqref{first} for some specific cocycle $\Phi$ over $\bbar \eta_t$ with  $\I(\Phi)
\neq 0.$ We choose 
$$\Phi(\wh f,t)\equiv \wh f(t).$$ Using Propositions \ref{p:is a cocycle}  and \ref{p:expectedvalofcocycles},  $\Phi$ is a cocycle over $\bbar \eta_t$, with integral
$\I(\Phi)
=1$.

\smallskip

First we prove \eqref{first}  for $\wh\nu$--a.e.~$\wh f$, i.e.~with
respect to this finite measure.  For our choice of $\Phi$, after a
change of variables, we have:
$$
\frac{1}{R} \int_0^R \frac{\wh f(e^u)}{e^{\alpha u}} \, \d u= \frac{1}{R} \int_0^R (\wh\tau_u \wh f) (1) \, \d u=
\frac{1}{R} \int_0^R \varphi  (\wh\tau_u \wh f) \, \d u,
$$
\noindent 
where $\varphi$ is the observable $\varphi(\wh f)\equiv \wh f(1);$
this is in $L^1(\wh D_{0>},\wh\nu)$ since
$$
\e(\varphi)= \int_{\wh D_{0>}}\varphi(\wh f) \, \d\wh\nu(\wh f)=
\int_{\wh D_{0>}} \wh f(1) \, \d\wh\nu(\wh f)=c_0
.$$ Accordingly, by the  Birkhoff Ergodic Theorem applied to $\wh\tau_t$, we have \eqref{first} for $\wh\nu$--a.e.~$\wh f$.

\smallskip

We shall prove that \eqref{first} also holds for $\bbar\nu$--a.e.~$\wh
f$. To this end, we claim that if \eqref{first} holds for some $\wh
f$, then it is true for all other points in the $\bbar \eta_t$-orbit
of $\wh f$. Indeed, for fixed $s\in\r$,  it is easily checked that
$\wh f(t)/(t-s)^{\alpha}$ and  $\wh
f(t)/t^{\alpha}$ share the same log average,  which is by hypothesis
$c_0\cdot \I(\Phi)$, hence
equivalently $\bbar\eta_s \wh f( t)/t^{\alpha}$ and $\wh
f(t)/t^{\alpha}$. 

Since the two flows
$(\wh D_{0>},\wh\nu,\wh\eta_t) $, $(\wh D_{0>},\bbar\nu,\bbar\eta_t) $
share a common cross--section, 
this yields \eqref{first} for $\bbar \nu$--a.e.~$\wh f$ and finishes the proof of Theorem \ref{t:logaverageself--sim}.\end{proof}

\subsection{The increment and renewal flows as stable horocycle flows}\label{ss:stable manifolds}

\smallskip

We have seen in Proposition \ref{p:commreln} that the scaling and increment 
flows obey the same commutation relation as  the geodesic and stable horocycle flows 
$g_t, h_t$ of a surface of constant negative curvature; here, making
use of the results of the last section, weakening
the definition of stable manifold allows us to make this analogy precise, whilst providing
tools needed for the proof 
of Theorem \ref{t:logaveragetheorem}.

\begin{defi}
Given a flow $g_t$ on a metric space $(X,d)$, preserving a measure
$\nu$, we shall call a
{\em stable horocycle flow} any flow $h_t$ 
whose orbit of   $\nu$--a.e.~$x\in X$ belongs to the $g_t$--stable manifold
$W^s(x)$. 
We say $h_t$ is a {\em Ces\`aro--average stable 
horocycle flow}
if the orbit of $\nu$--a.e.~$x\in X$ belongs to $W^s_{CES}(x)$ (as defined in
 Theorem 
\ref{t:main2}).
\end{defi}

Parts $(i)$ and $ (iv)$ of the next proposition show the increment flow of the
Mittag--Leffler process is a  
Ces\`aro--average horocycle flow for the scaling flow,
while parts $(ii), (iv)$ show something similar for the renewal flow. Part 
$(iii)$ is a related statement   which we use in  the proof
of Theorem \ref{t:logaveragetheorem}:
\begin{prop}\label{p:stable manifolds}(Increment and renewal flows as
  Ces\`aro--average horocycle flows for $\wh \tau_t$)
\item{(i)} The orbit 
$\{\bbar\eta_r \wh Z\}_{r\in\r}$ is in $W^s_{CES}(\wh Z)$ as for $\wh \nu$--a.e.~path $\wh Z$ then 
for any $r\in \r$ fixed, 
\[
\lim_{T\to\infty}\frac{1}{T} \int_0^T |{\wh\tau}_t {\wh Z} (1) - {\wh \tau}_t (\bbar\eta_r {\wh Z})(1)| \; \dd t = 0.
\]
\item{(ii)}  With respect to the joining $\wh \mu$ of Theorem \ref{t:main2}, for a.e.~pair 
$(\wh Z, \bbar N)$, 
the orbit $\{\bbar\eta_r (h\circ \bbar N)\}_{r\in\r}$ is a subset of  $W^s_{CES}(\wh Z)$ since for $\wh \mu$--a.e.~path $h\circ \bbar N$, then
for any $r\in \r$ fixed, 
\[
\lim_{T\to\infty}\frac{1}{T} \int_0^T |{\wh\tau}_t {\wh Z} (1) - {\wh \tau}_t (\bbar\eta_r (h\circ \bbar N))(1)| \; \dd t = 0.
\]
\item{(iii)} Moreover, for $\bbar
\mu$--a.e.~path $\bbar N$, we have for each $r\in\r$ fixed:
\[
\lim_{T\to\infty} \frac{1}{\log T} \int_1^T |\bbar\eta_r(\bbar N) (t) -  \bbar N(t)| \; \frac{\dd t}{\wh a(t)\; t}= 0.\]
\item{(iv)} All the statements remain true with $|\cdot|$ replaced by
  $||\cdot||_{[0,1]}^\infty$. 

\end{prop}
\begin{proof}
We know that $\Phi(\wh Z, t)\equiv \wh Z(t)$ defines a
cocycle over the increment flow $\bbar\eta_t$.  Now since $\wh Z$ is
nondecreasing, for fixed $r$ and large $t$ we have
 \[|\wh \tau_t (\bbar\eta_r \wh Z)(1)-\wh\tau_t \wh Z (1)|= e^{-\alpha
 t}|\wh Z
 (e^t+r)-\wh Z(r)-\wh Z(e^t)|\leq e^{-\alpha t}(\wh Z
 (e^t+r)-\wh Z(e^t) )+ e^{-\alpha t} \, \wh Z (r).\]
So by a logarithmic change of variables, proving $(i)$ amounts to
checking that for large $T_0$ 
\[
\lim_{T\to\infty}\frac{1}{\log T} \int_{T_0}^T  (\wh Z(t+r) -\wh Z
(t))\; \frac{\d t}{t^{\alpha +1}}=0, \;\;\;\;\wh\nu-{\text{a.s.}}\]
\noindent
Following the reasoning at the end of the proof of Theorem
\ref{t:logaverageself--sim}, this holds true.
 
\smallskip
For part ($ii)$, as above, we have that
for fixed $r$ and for large $t$:
\[
|(\bbar\eta_r  (h\circ \bbar N))(t)- \wh Z(t) | \leq 
|h\circ \bbar N(t+r) - \wh Z(t+r)|+ \wh Z(t+r)-\wh Z(t) +h\circ \bbar N (r).
\]
Using \eqref{main} and then following exactly the strategy used in
the proof of $(i)$,  one shows $(ii)$. 

\smallskip

The proof of part $(iii)$ follows the same
pattern as that of $(ii)$: from Corollary \ref{c:ChungErd}, the log average of $\bbar N(t)/\wh
a(t)$ equals $\e(\wh Z(1))$ and,  by regular variation $\wh
a(t-r)\sim \wh a(t)$ for $t$ large.

Lastly, as in the proof of Theorem \ref{t:main2}, these  results also
hold for $||\cdot||_{[0,1]}^\infty$: since $h(\cdot)$, $\bbar N$ and $\wh Z$ are nondecreasing, 
the same proofs go through.
\end{proof}

\subsection{Order--two ergodic theorems for the Mittag--Leffler and
  renewal flows and the renewal transformation: proofs of Theorem \ref{t:logaveragetheorem} and
  Corollary \ref{c:logaveragerentf}}\label{ss:proofofthm}

\begin{proof}[Proof of Thm.~\ref{t:logaveragetheorem}] Part $(i)$ follows
  directly from Thm.~\ref{t:logaverageself--sim}, since from
  Prop.~\ref{p:nextcommreln} the two--sided  Mittag--Leffler process is
  dual to an ergodic self--similar process with stationary ergodic increments. 
 
\smallskip

For part $(ii)$, 
part of the strategy for the proof  of Thm.~\ref{t:logaveragetheorem} still applies. First, by the Hopf theorem, it is sufficient to prove the statement for a specific cocycle. We choose
$\Phi(\bbar N, t)= \bbar N(t)$. For $\wh\mu-$a.e.~$\bbar N$,  statement \eqref{e:logAv} holds
by Corollary \ref{c:ChungErd}.
Part $(iii)$ of Prop.~\ref{p:stable manifolds} then shows the same statement is true for any path in the 
$\bbar\eta_t$- orbit of $\bbar N$. Now since the flows $\bbar\eta_t$ and $\wh\eta_t$ 
share a common cross--section, having the statement for a.e.~point $\bbar N$ with respect to the finite measure $\wh\mu$
implies this for the infinite measure $\bbar\mu$ as well.
We are done with the proof of \eqref{e:logAv} and hence of Thm.~\ref{t:logaveragetheorem}.
\end{proof}

Next we move to discrete time and  the:
\begin{proof}[Proof of Corollary \ref{c:logaveragerentf}]  
By the Markov property,
the return times to the set $A$ under the shift map are 
 i.i.d. Therefore, for $F$  the   distribution function of the return times, 
$N_n$ is a renewal process; indeed $N_n= \bbar N(n)$
where $\bbar N(t)$ is the continuous-time renewal process for the gap distribution $\mu_F$.

We  prove
  that $(a)$ is equivalent to $(b)$. That $F$ is in the domain of
  attraction of $G_{\alpha}$ is equivalent to $1-F$ regularly varying of index 
$-\alpha$. This, in turn, is equivalent to saying that the renewal function
$U(t)=\sum_{k\geq 0} F^{k\star}(t)=\e(\bbar N(t)+1)$, and hence the return sequence
$(\bbar a_n)$, is regularly varying of index  $\alpha$;
see p.~361 of 
\cite{BinghamGoldieTeugels87}. 

\smallskip

Proof of $(i)$:  Assuming $(b)$,  the function $\varphi=\chi_{A}$
generates a cocycle $\Psi(x,n)$ by equation 
\eqref{e:cocyclesum}.
Applying  
$(ii)$ of
Theorem \ref{t:logaveragetheorem} to $\Psi(x,t)$, 
$(i)$ of the corollary holds for $\Psi(x,n)$. By the  Hopf Ratio
Ergodic Theorem this then passes  to any
other cocycle  over the transformation.

For an alternative argument, the process $Y_i$ is a factor of the Markov
shift, and  is the event process
model of a renewal transformation. The renewal shift model is a
countable state Markov chain, so the argument just given applies to
prove 
the order--two ergodic theorem  for the renewal
transformation. Now via the Ratio
Ergodic Theorem this fact lifts to any conservative ergodic  map which
factors onto it, in  particular,  to the Markov shift. (More generally
this passes on to any {\em similar} transformation; see Corollary 6 of \cite{Fisher92}).

The convergence of the Chung-Erd\"os averages of part $(ii)$ are
equivalent to the log averages of part $(i)$ in this case, as mentioned in Proposition 1 of \cite{AaronsonDenkerFisher92}.

\medskip

We now prove $(iii)$. Setting $\wh a_n\equiv\frac{1}{1-F(n)}$, Feller
\cite{Feller49} (Thm.~7) proved
that $N_n/\wh a_n$ converges in law to a Mittag--Leffler distribution ${\cal M}_{\alpha}$. Again, from p.~361 of \cite{BinghamGoldieTeugels87}, we know
that $\bbar a_n\sim {\mathtt c}\, \wh a_n$, where
${\mathtt c}=(\Gamma(1-\alpha)\Gamma(1+\alpha))^{-1}=\e(\wh Z(1))$;
see  (7.2) on p.~113 of \cite{Feller49} for the value of $\e(\wh
Z(1))$. 

Thus
$N_n/\bbar a_n$ converges to a rescaled Mittag--Leffler
distribution of index $\alpha$, of mean  1.

\smallskip

 Now, $F$
belongs to the domain of attraction of $G_{\alpha}$ so $S(n)/a(n)
\stackrel{law}{\longrightarrow} G_{\alpha}$, where 
$S(\cdot)$ is the polygonal interpolation extension of $(\bbar S_n)$. Setting $N\equiv S^{-1}$
and $\wh a\equiv a^{-1}$, $\forall x>0$ we have $\P(S(n)\leq x
a(n))=\P(N(xa(n))\geq n)$ and $\P(N(y)/\wh a(y) \geq \wh a(y/x)/\wh
  a(y))\to G_{\alpha}(x)$ as $y\to \infty$, with $y\equiv x
a(n)$. As $\wh a$ is regularly
varying of index  $\alpha$, $\forall x\in (0,b], b>0$, $\wh a(y/x) / \wh a(y)$ converges uniformly
to $x^{-\alpha}$; see  p.~22 of
\cite{BinghamGoldieTeugels87}. So $N(n)/\wh a(n)
\stackrel{law}{\longrightarrow} {\cal M}_{\alpha}$, thus $\wh a
(n)\sim \wh a_n$ which yields $\bbar a_n \sim {\mathtt c} \wh a (n)$.
\end{proof}

\begin{rem}Both the  renewal process
$N_n$ and the cocycle $\Psi(x,n)$, which defines the 
 increment shift model of the renewal
transformation, $\wt N_n$, occured in the proof, but 
these are not quite the same:
they differ by a time shift of the event process. 
Indeed, defining $\Psi^\circ(x,n)= \Psi(\sigma x, n)$, then
 $\Psi^\circ$ is also a
cocycle over the renewal transformation, and  from
\eqref{eq:N_n}, $N_n= \Psi^\circ(x,n)$. 
\end{rem}

\subsection{Identification of the constant $\mathtt c$: proof of Proposition \ref{prop1}}
\label{s:range}

We recall the notation regarding stable laws and processes from the
beginning of \S \ref{1}. The law $G_\alpha=G_{\alpha,
  1}$ has an  
especially nice Laplace transform:
\begin{equation}\label{c}
\e(e^{-wX})= \exp( -\check c_{\alpha} \, w^\alpha), \;\;\;\;\;  \check c_\alpha= \frac{\Gamma(3-\alpha)}{\alpha(1-\alpha)}.\end{equation}

\smallskip

Since the corresponding stable process $Z$ has increasing paths with a dense set of jump points,
then  $\wh Z$ has  a nowhere dense set $C_{\wh Z}$ of points of increase, 
with a flat stretch of $\wh Z$  (on the gaps of $C_{\wh Z}$) 
corresponding to each of the  jumps of $Z$. 

\smallskip

We recall the definition of 
Hausdorff $\varphi$--measure $H_\varphi$ (see \cite{Falconer85}, \cite{Mattila95}):
\begin{defi}
Let $\varphi\in D_{0>}(\R^+)$, with $d$ the Euclidean metric in $\r^n$.
A $\delta$-cover of $A\subseteq \r^n$ 
is a countable cover by subsets $E_i$ of diameter $|E_i|<\delta$.
Then
\[
H_\varphi(A)= \lim_{\delta\to0} \left(\inf_{\{E_i\}: \, \delta-cover\, of\, A}\sum_{i=1}^\infty \varphi(|E_i|)\right).\]
\end{defi}

We suppose the {\em gauge function} $\varphi$ is 
{\em regularly varying of index $\alpha$ at zero}; that is, for some 
$\alpha>0$,  for each $a>0$,
$\varphi(at)/\varphi(t)\stackrel{t \to 0} {\longrightarrow}a^\alpha$. 
The following scaling property, which expresses the self--similarity of the measure,  follows
from the regular variation of the gauge function, as noted in \cite{BedfordFisher92}: 
\begin{lem}\label{fi}
If $\varphi$ is regularly varying of index $\alpha$ at zero, then for every $a>0$,
\[H_\varphi(aA)= a^\alpha H_\varphi(A).
\]
\end{lem}

\begin{defi}
Given a stochastic process $X(t)$ with paths in $D$, the {\em range} of $X$ is the 
set--valued process given by taking the image of intervals, as follows:
\[
R_X(t)=\left\{\begin{array}{ll}
X([0,t]) &\mbox{for $t\geq 0$}\\
X([t,0]) &\mbox{for $t< 0$.}
\end{array}\right.\]
\end{defi}

For $\alpha\in (0,1)$ we define 
\[\sbar Z(t)= \wh c_\alpha^{1/\alpha} Z(t),\;\;\;\; \text {where}\;\;\;\; \wh c_\alpha=1/\check c_{\alpha}\;\; \text {defined\, in} \, \eqref{c},\]
\[ 
\psi(t)= t^\alpha \, (\log \log \frac{1}{t})^{1-\alpha} , \, \text{and} \, \;\; c_\alpha= \frac{\wh c_\alpha}{\tilde c_\alpha} = \frac{\alpha^{1-\alpha}(1-\alpha)^\alpha}{\Gamma(3-\alpha)},\; \text{as}\, \text{in}\, \eqref{e:Hawkesgauge}.
\]

\begin{lem}\label{p:identifyingconstant}${(i)}$ (Hawkes) For $\nu$--a.e.~$Z$, for all $t\geq 0$, 
\[ 
H_\psi(R_{\sbar Z}(t))= \tilde c_\alpha \cdot t,\;\;\;\;\;\;\;{\text with}\;\;\; 
\tilde c_\alpha= \alpha^\alpha (1-\alpha) ^{1-\alpha}.
\]
\item{(ii)} We have, for $\wh\nu$--a.e.~$\wh Z$:
\[ 
 c_\alpha H_\psi(C_{\wh Z}\cap [0, T])= \wh Z(T).
\]
\end{lem}

\begin{proof} Using \eqref{c}, $\sbar Z$ satisfies 
\[
\e(e^{-\omega \sbar Z (1)})=\e(e^{-\omega\, \wh c_{\alpha}^{1/\alpha}  Z (1)})= \exp(-\omega^{\alpha}
\, \wh c_{\alpha}\, \check c_{\alpha})=e^{-\omega^{\alpha}},\]
and then part $(i)$ is Theorem 2 of \cite{Hawkes73}.
\smallskip

To deduce $(ii)$ from this, we denote by $\wh {\sbar Z}$  the generalized inverse of $\sbar Z$ so $R_{\sbar Z}( \wh {\sbar Z}(T))
= C_{\wh {\sbar Z}}\cap [0, T]$. On the other hand, as $\wh {\sbar Z}(t)= \wh Z(t/ \wh c_{\alpha}^{1/\alpha})$, $C_{\wh {\sbar {Z}}}= \wh c_{\alpha}^{1/\alpha}\, C_{\wh {Z}}$.

\smallskip

We then  set $t= \wh {\sbar Z}(T)$ and $(i)$ says that 
\[ 
H_{\psi}(R_{\sbar Z}( \wh {\sbar Z}(T))= H_\psi(C_{\wh{\sbar Z}}\cap [0, T])=  H_\psi( \wh c_{\alpha}^{1/\alpha}\, (C_{\wh {Z}}\cap [0, T/ \wh c_{\alpha}^{1/\alpha}]))= \tilde c_\alpha \, \wh{\sbar Z}(T)=\tilde c_\alpha \, \wh Z(T/ \wh c_{\alpha}^{1/\alpha}).\]

So $\forall T>0$, since $\psi$ is regularly varying of index $\alpha$
at zero, by Lemma \ref{fi} we prove $(ii)$:
\[
H_\psi( \wh c_{\alpha}^{1/\alpha}\, (C_{\wh {Z}}\cap [0, T]))= \wh c_{\alpha} H_\psi( C_{\wh {Z}}\cap [0, T])=\tilde c_\alpha \, \wh{\sbar Z}(T).\]
\end{proof}

We are now ready for the

\begin{proof}[Proof of Proposition \ref{prop1}] Since the scaling flow for $\wh Z$ is ergodic, 
$c_\alpha$ times 
the right order--two density at $x= 0$ is, using $(ii)$ and 
Birkhoff's ergodic
 theorem,
and defining $f\in L^1(\wh D_{0>}, \wh \nu)$ by $f(\wh Z)= \wh Z(1)$, 
\[
\begin{aligned}
\lim_{T\to\infty}\frac{1}{T}\int_0^T\frac{c_\alpha H_\psi(C_{\wh
    Z}\cap [x, e^{-s}])}{e^{-\alpha s}}\d s&=
\lim_{T\to\infty}\frac{1}{T}\int_0^T\frac{\wh Z( e^{-s})}{e^{-\alpha s}}\d s=
\lim_{T\to\infty}\frac{1}{T}\int_0^T(\wh \tau_{-s}\wh Z)(1)\d s=\\
\lim_{T\to\infty}\frac{1}{T}\int_0^Tf(\wh \tau_{-s}\wh Z)\d s&=
\int_{\wh D_{0>} }f\,  \d\wh \nu=   \e(\wh Z(1))\equiv \c=\frac{1}{\Gamma(1-\alpha)\Gamma(1+\alpha)}=\frac{\sin
\pi\alpha}{\pi \alpha},
\end{aligned}
\]
\noindent 
see \cite{BinghamGoldieTeugels87}, p.361. Thus for $\wh \nu$--a.e.~$\wh Z$, the right order--two density at $x= 0$
exists and equals $c_\alpha^{-1} {\mathtt c}$.
This is not enough: we want to show this holds at $H_\psi$--a.e.~$x$ in $C_{\wh Z}$. The proof is a 
Fubini's theorem argument.
Thus, writing 
  $\wh D_1\subseteq D$ 
for  the set of all $\wh Z$ such that the right order--two density
at $0$ equals $c_\alpha^{-1} {\mathtt c}$, 
we let $D_1$ denote the corresponding set of paths $Z= \wh Z^{-1}$.
We have just seen that 
$\wh \nu(\wh D_1)=1$, so equivalently $\nu(D_1)=1$. Without loss of generality 
we can take $D_1$ to be Borel measurable, since it contains a Borel subset also of full measure (by Prop.28, Chapter 12 of \cite{Royden68}).

From Proposition \ref{p:firstcommreln}, we know $\eta$ is jointly Borel measurable.
Now considering, for $T>0$, 
\[A_T=
\{(t, Z)\in [0,T]\times D:\, \eta_t Z\in D_1\}
;\] this is a measurable set for the 
product of the Borel $\sigma$--algebras, hence is measurable in the $(m\times \nu)$-completion  where $m$ is Lebesgue measure. Since $\eta_t$ preserves $\nu$ (Proposition \ref{p:nextcommreln}),
we know that for each $t\in [0,T]$, 
$\nu(\{Z:\, (t, Z)\in A_T\})=\nu(\eta_t^{-1}(D_1))=1$. By Fubini's theorem therefore, for $\nu$--a.e.~$Z$, 
$m(\{t:\, (t, Z)\in A_T\})=T$. 

Now  for $\nu$--a.e.~$Z$, 
$m$ pushes forward by $Z$ to the restriction  of $H_\psi$ to $C_{\wh Z}\cap
[0,T]$. Hence in conclusion for $\wh \nu$--a.e.~$\wh Z$,  
we know that for 
$H_\psi$--a.e.~$x$ in $C_{\wh Z}\cap [0,T]$ for any $T>0$, the right order--two density at $x$
exists and equals $c_\alpha^{-1} {\mathtt c}$. (Taking countable intersections over $T=1,2,\dots$, we can replace 
$C_{\wh Z}\cap [0,T]$ by  $C_{\wh Z} $).
\end{proof}

\begin{rem}\label{r:IntegerHausdorff meas}
The constant $\c$ is also 
the order--two density at $+\infty$ (rather than at $0$) both of $C_{\wh Z}$ and of the  
integer fractal set $\Cal O$ of renewal events. 
We  can think of this second limit as defining a  finitely additive Hausdorff measure:
 taking a long interval $[0,T]$, we cover the points in 
$\Cal O\cap [0,T]$ with intervals of length $1$ and then sum them up 
with regularly varying gauge function 
$\phi(r)= 1/\wh a(1/r)$;  the integer Hausdorff $\phi$--measure 
of $\Cal O$ is then equal to  $\c$.
\end{rem}

\enddocument